%% file: arxiv.tex
\documentclass[a4paper,10pt,twoside,onecolumn,final]{article}
\input{preamble}
\input{arxiv_preamble}

\begin{document}

\maketitle

\input{contents/00_frontmatter}

\input{contents/01_introduction}

\input{contents/02_preliminaries}

\input{contents/03_poisson}

\input{contents/04_elasticity}

\input{contents/05_stokes}

\input{contents/06_numerical_experiments}

\input{contents/07_conclusion}

\section*{Acknowledgments}
This research was supported by the Swiss National Science Foundation via project MINT n. 200021\_215099, PDE tools for analysis-aware geometry processing in simulation science.

\appendix
\input{contents/A_results_on_trace_spaces}

\input{contents/A_linear_elasticity}

\input{contents/A_stokes}

\bibliographystyle{ws-m3as}
\bibliography{references}

\end{document}

%% file: preamble.tex
\usepackage{lipsum}
\usepackage{amsfonts}
\usepackage{graphicx}
\usepackage{epstopdf}
\usepackage{algorithmic}
\usepackage{mathtools}
\usepackage[hidelinks]{hyperref}
\hypersetup{
    colorlinks=false,
    pdftitle={A Certified Goal-Oriented A Posteriori Defeaturing Error Estimator for Elliptic PDEs},
    pdfauthor={P. Weder and A. Buffa}
    }
\usepackage{subcaption}
\usepackage{comment}
\usepackage{bbm}
\usepackage{xcolor}
\usepackage{tikz}
\usepackage{tikz-3dplot}
\usetikzlibrary{fadings}
\usetikzlibrary{patterns}
\usetikzlibrary{shadows.blur}
\usetikzlibrary{shapes}
\usetikzlibrary{calc}
\usetikzlibrary{positioning}
\usetikzlibrary{hobby}
\usetikzlibrary{patterns.meta, arrows.meta}
\usetikzlibrary{backgrounds}

\usepackage[super,ref]{cite}
\usepackage{caption}
\ifpdf
  \DeclareGraphicsExtensions{.eps,.pdf,.png,.jpg}
\else
  \DeclareGraphicsExtensions{.eps}
\fi

\usepackage{enumitem}
\setlist[enumerate]{leftmargin=.5in}
\setlist[itemize]{leftmargin=.5in}

\usepackage{amsopn}
\usepackage{macros}

\usepackage{cleveref} 
\usepackage{etoolbox} 

\gappto{\appendix}{
  \crefalias{section}{appendix}

  \renewcommand{\thesection}{\Alph{section}}
  \renewcommand{\thesubsection}{\Alph{section}.\arabic{subsection}}
  
}
\crefname{figure}{Fig.}{Figs.}
\Crefname{figure}{Fig.}{Figs.}

\crefname{table}{Table}{Tables}
\Crefname{table}{Table}{Tables}

\crefname{equation}{}{}
\Crefname{equation}{}{}

\crefname{lemma}{Lemma}{Lemmas}
\Crefname{lemma}{Lemma}{Lemmas}

\crefname{theorem}{Theorem}{Theorems}
\Crefname{theorem}{Theorem}{Theorems}

\crefname{appendix}{Appendix}{Appendices}
\Crefname{appendix}{Appendix}{Appendices}

\crefname{assumption}{Assumption}{Assumptions}
\Crefname{assumption}{Assumption}{Assumptions}

\crefname{definition}{Definition}{Definitions}
\Crefname{definition}{Definition}{Definitions}

%% file: arxiv_preamble.tex
\usepackage{geometry}
\usepackage{authblk}
\usepackage{amsthm,amssymb}
\usepackage{aliascnt} 

\newtheorem{theorem}{Theorem}[section]

\newaliascnt{lemma}{theorem}            
\newtheorem{lemma}[lemma]{Lemma}        
\aliascntresetthe{lemma}                

\newaliascnt{proposition}{theorem}
\newtheorem{proposition}[proposition]{Proposition}
\aliascntresetthe{proposition}

\newaliascnt{corollary}{theorem}

\aliascntresetthe{corollary}

\newaliascnt{remark}{theorem}
\newtheorem{remark}[remark]{Remark}
\aliascntresetthe{remark}

\newaliascnt{definition}{theorem}
\newtheorem{definition}[definition]{Definition}
\aliascntresetthe{definition}

\newtheorem{assumption}{Assumption}[section]

\usepackage{cleveref}

\crefname{appendix}{Appendix}{Appendices}
\Crefname{appendix}{Appendix}{Appendices}

\gappto{\appendix}{
  \crefalias{section}{appendix}
  
  \renewcommand{\thesection}{\Alph{section}}
  \renewcommand{\thesubsection}{\Alph{section}.\arabic{subsection}}
  
}

\crefname{figure}{Fig.}{Figs.}
\Crefname{figure}{Fig.}{Figs.}

\crefname{equation}{}{}
\Crefname{equation}{}{}

\crefname{lemma}{Lemma}{Lemmas}
\Crefname{lemma}{Lemma}{Lemmas}

\crefname{assumption}{Assumption}{Assumptions}
\Crefname{assumption}{Assumption}{Assumptions}

\crefname{definition}{Definition}{Definitions}
\Crefname{definition}{Definition}{Definitions}

\crefname{theorem}{Theorem}{Theorems}
\Crefname{theorem}{Theorem}{Theorems}

\crefname{proposition}{Proposition}{Propositions}
\Crefname{proposition}{Proposition}{Propositions}

\crefname{corollary}{Corollary}{Corollaries}
\Crefname{corollary}{Corollary}{Corollaries}

\crefname{remark}{Remark}{Remarks}
\Crefname{remark}{Remark}{Remarks}

\makeatletter
\renewcommand\@maketitle{%
  \newpage
  \null
  \vspace{1em}%
  \begin{center}%
  \let \footnote \thanks
    {\normalsize \bfseries \@title \par}%
    \vspace{1.5em}%
    {\normalsize
      \lineskip .5em%
      \begin{tabular}[t]{c}%
        \@author
      \end{tabular}\par}%
    \vspace{1em}%
    {\small \itshape $^{1}$Institute of Mathematics, Ecole Polytechnique Fédérale de Lausanne, \\
    Station 15, CH-1015, Lausanne, Vaud, Switzerland \par}%
    \vspace{0.5em}%
    {\small \texttt{$^{\dagger}$philipp.weder@epfl.ch, $^{*}$annalisa.buffa@epfl.ch} \par}%
    \vspace{1em}%
    {\normalsize \@date}%
  \end{center}%
  \par
}
\makeatother

\title{A Certified Goal-Oriented A Posteriori Defeaturing \\ Error Estimator for Elliptic PDEs}
\author{Philipp Weder$^{1,\dagger}$ and Annalisa Buffa$^{1, *}$}
\date{\today}

\usepackage[font=footnotesize, labelsep=colon]{caption}
\makeatletter
\newcommand{\bibfont}{\fontsize{9}{11}\selectfont}
\renewcommand\@biblabel[1]{#1.}

\makeatother

\usepackage{titlesec}

\titleformat{\section}
  {\normalfont\bfseries} 
  {\thesection.}         
  {0.5em}                
  {}                     

\titleformat{\subsection}
  {\normalfont\bfseries\itshape} 
  {\thesubsection.}              
  {0.5em}
  {}

\titleformat{\subsubsection}
  {\normalfont\itshape}          
  {\thesubsubsection.}           
  {0.5em}
  {}

\titlespacing*{\section}{0pt}{18pt plus 3pt minus 6pt}{5pt}
\titlespacing*{\subsection}{0pt}{18pt plus 3pt minus 6pt}{5pt}
\titlespacing*{\subsubsection}{0pt}{18pt plus 3pt minus 3pt}{5pt}

%% file: contents/00_frontmatter.tex
\begin{abstract}
Defeaturing, the process of simplifying computational geometries, is a critical step in industrial simulation pipelines for reducing computational cost. Rigorous a posteriori estimators exist for the global energy-norm error introduced by geometry simplifications. However, practitioners are usually more concerned with the accuracy of specific quantities of interest (QoIs) in the solution. This paper bridges that gap by developing mathematically certified, goal-oriented \emph{a posteriori} defeaturing error estimators for Poisson's equation, linear elasticity, and Stokes flow. First, we derive new reliable energy-norm estimators for features subject to Dirichlet boundary conditions in linear elasticity and Stokes flow, based on existing results for Poisson's equation. Second, we formulate general energy-norm estimators for multiple negative features, subject to either Dirichlet or Neumann boundary conditions for the first time. Finally, we combine these estimators with the dual-weighted residual (DWR) method to obtain reliable estimates for linear QoIs and demonstrate their effectiveness across a range of numerical experiments.
\end{abstract}

%% file: contents/01_introduction.tex
\section{Introduction}
Solving partial differential equations (PDEs) in industrial applications often involves geometrically complex domains. These geometries contain features across multiple scales, which makes meshing difficult and simulations computationally expensive. Even the presence of a single, relatively small geometric feature can increase the computational cost of a simulation by an order of magnitude due to the need for additional degrees of freedom.\cite{white_meshing_2003,lee_small_2005} It is standard practice to simplify these geometries by removing small or irrelevant features, a process known as defeaturing.

Traditionally, defeaturing relies on \emph{a priori} criteria that are applied before any simulation is performed. These criteria are often heuristic, ranging from an engineer's subjective expertise\cite{lee_development_2003} to simplifications based on purely geometrical rationale\cite{thakur_survey_2009} or prior knowledge of the underlying physical problem.\cite{fine_automated_2000,foucault_mechanical_2004,rahimi_cad_2018}

For a certified modeling and simulation process, \emph{a posteriori} defeaturing error estimators are necessary. Early approaches provided the crucial insight that the error concentrates at simplified boundaries\cite{ferrandes_posteriori_2009} or were based on asymptotic considerations such as feature sensitivity analysis (FSA), which crucially rely on the assumption of infinitesimally small features.\cite{gopalakrishnan_formal_2007,gopalakrishnan_feature_2008,turevsky_defeaturing_2008,turevsky_efficient_2009} More recent methods using the dual-weighted residual (DWR) technique\cite{becker_optimal_2001,prudhomme_goal-oriented_1999} or the reciprocal theorem have also been explored.\cite{li_estimating_2011,li_goal-oriented_2013,tang_evaluating_2013,zhang_estimation_2016} However, these works lack a rigorous mathematical certification and involve heuristic parameters that depend on the feature's size.

This work builds on the framework for \emph{a posteriori} defeaturing error estimation introduced by Buffa et al.\cite{buffa_analysis-aware_2022}, where the defeaturing problem is rigorously analyzed for the Poisson problem with one feature subject to Neumann boundary conditions. It provides estimates that explicitly capture the size of the feature, allowing for the treatment of features of any size.
Subsequent work extended this framework in several directions: Antolín and Chanon\cite{antolin_analysisaware_2024} generalized the estimates to multiple features with Neumann boundary conditions, as well as to the linear elasticity and Stokes' equations. Weder and Buffa\cite{weder_analysis-aware_2025} treated the case of negative features subject to Dirichlet boundary conditions for Poisson's equation. Furthermore, Buffa et al.\cite{buffa_adaptive_2024,buffa_equilibrated_2024,buffa_adaptive_2025} combined the defeaturing estimators with \emph{a posteriori} numerical error estimators for isogeometric analysis \cite{hughes_isogeometric_2005,marussig_review_2018} and finite elements. Chanon\cite{chanon_adaptive_2022} provides a comprehensive treatment of the topic.

While the previously introduced estimators provide robust estimates of the defeaturing error in the global energy-norm of the underlying PDE problem\cite{buffa_analysis-aware_2022,antolin_analysisaware_2024,chanon_adaptive_2022}, engineers are often interested in the accuracy of specific quantities of interest (QoIs), such as the average stress on a surface, the flux across a boundary, or the displacement at a specific point. In the latter cases, the global energy norm may not be an accurate proxy for the error in these QoIs.
To bridge this gap, we combine the rigorous energy-norm estimators with DWR.\cite{becker_optimal_2001,prudhomme_goal-oriented_1999}

This paper makes three primary contributions: First, we extend the analysis of Dirichlet features to linear elasticity and Stokes flow by deriving new reliable energy-norm estimators, which are of independent interest. Second, we present a unified problem formulation for geometries with mixed Dirichlet and Neumann features. Building on this foundation, we then develop certified goal-oriented defeaturing estimators for linear quantities of interest applicable to all three model problems. Our approach advances beyond previous DWR-based methods, which rely on heuristic approximations or the dismissal of uncomputable error terms\cite{li_estimating_2011,li_goal-oriented_2013}. In contrast, our framework provides a mathematically certified error bound that is explicit in the feature's size. Moreover, it is significantly more general, providing estimators for complex configurations involving multiple negative isotropic features, features on the domain boundary, and those subject to either Dirichlet or Neumann conditions, thus overcoming key limitations of prior work.\cite{li_goal-oriented_2013}

The paper is structured as follows: \Cref{sec:preliminaries} provides the necessary geometric and functional analytic background. \Cref{sec:poisson,sec:elasticity,sec:stokes} are dedicated to the Poisson, linear elasticity, and Stokes problems, respectively. For each model, we first formulate the general defeaturing problem for geometries with multiple features and mixed Dirichlet and Neumann boundary conditions, and then derive the corresponding goal-oriented \emph{a posteriori} error estimators. The reliability proofs for the new energy-norm estimators for Dirichlet features in elasticity and Stokes flow are detailed in \cref{sec:linear elasticity:reliability,sec:stokes:reliability}, respectively. Finally, we validate our theoretical results with numerical experiments in \cref{sec:numerical experiments} and present our conclusions in \cref{sec:conclusion}.

%% file: contents/02_preliminaries.tex
\section{Preliminaries}
\label{sec:preliminaries}
This section provides the necessary background for our analysis, beginning with the terminology used to describe complex geometries with negative features. It then defines the functional spaces in which our analysis is set.

\subsection{Domains, boundaries, and features}
\label{subsec:domains boundaries features}

\begin{figure}
    \centering
    \begin{subfigure}[B]{0.45\textwidth}
    \centering
    \input{tikz/intro_internal}
    \caption{Internal negative feature.}
    \label{fig:preliminaries:internal feature}
    \end{subfigure}
    \begin{subfigure}[B]{0.45\textwidth}
    \centering
    \input{tikz/intro_boundary}
    \caption{Negative feature on the boundary.}
    \end{subfigure}
    \caption{Domain with an internal feature $\feat$ (a) and a boundary feature $\feat$ (b). The defeatured and simplified boundaries are denoted by $\defbd$ and $\simpbd$, respectively. The dotted parts represent the feature's interior.}
    \label{fig:preliminaries:feature types}
\end{figure}

In what follows, let $\domain \subset \R^n$ denote a bounded open set. We say that $\domain$ is a Lipschitz domain if its boundary $\partial \domain$ is locally the graph of a Lipschitz function. Unless otherwise stated, we restrict our analysis to such domains for $n=2$ or $n=3$. We write $B_r(x) \subset \R^n$ for the open ball of radius $r$ centered at $x$ with respect to the Euclidean norm.

Let $\genman$ be a $d$-dimensional submanifold of $\R^n$ with $d \in \{n, n-1\}$. If $d = n$, then $|\genman|$ denotes the Lebesgue measure of the set $\genman$, while for $d = n-1$, $|\genman|$ denotes the $(n-1)$-dimensional Hausdorff measure. The closure and interior of $\genman$ are denoted by $\overline{\genman}$ and $\interior{\genman}$, respectively. We denote by $\diam{\genman}$ the Euclidean diameter of $\genman$ in $\R^n$. Furthermore, we denote by $c_\genman \in \R^n$ and $\rcirc(\genman) > 0$ the circumcenter and circumradius of $\genman$, such that the latter is the radius of the smallest open ball in $\R^n$ containing $\overline{\genman}$. Finally, we will require the geometrical features defined below to satisfy the following isotropy property:

\begin{definition}
\label{def:isotropic family}
Let $\mathcal{F}$ be a family of connected $d$-dimensional Lipschitz submanifolds with or without boundary of $\R^n$, for $d \in \{n-1, n\}$. We say that $\mathcal{F}$ is a \emph{uniformly isotropic family} of features if there exists a constant $C_{\rm iso} > 0$ such that
\begin{align*}
    \diam{\genman}^d \leq C_{\rm iso} |\genman| \quad \forall \genman \in \mathcal{F}.
\end{align*}
\end{definition}

Note that the connectedness assumption in \cref{def:isotropic family} confines the isotropy constant to a single feature's intrinsic geometric properties, so that separated, disconnected components cannot inflate it artificially.

\begin{assumption}
\label{asspt:isotropy}
Subsequently, whenever we consider an $(n-1)$-dimensional Lipschitz submanifold of $\R^n$, e.g., a piece of the boundary of a Lipschitz domain, we assume that it belongs to a uniformly isotropic family in the sense of \cref{def:isotropic family}, with a fixed constant $C_{\rm iso} > 0$.
\end{assumption}

In light of \cref{asspt:isotropy}, we introduce the following notation: We write $A \lesssim B$ whenever $A \leq c B$ for some constant $c$ independent of the size of the considered domains and boundaries, but that possibly depends on the dimension $n$, 
the isotropy constant $C_{\rm iso}$, and the Lipschitz characters of the involved domains. Additionally, we use $A \simeq B$, whenever $A \lesssim B$ and $B \lesssim A$.

Next, we will formalize the idea of a feature. Let $\domain \subset \R^n$ be a Lipschitz domain. As illustrated in \cref{fig:preliminaries:feature types}, we say that $\feat \subset \R^n$ is a negative feature of $\domain$ if $\overline{F} \cap \overline{\domain} \subset \partial \domain$. In the following, we consider domains with arbitrary, but finite, numbers $N_f \in \N$ of negative features. We denote the set of all features by $\featset = \{\feat^1, \ldots, \feat^{N_f}\}$. Hereby, we assume that the features are \emph{separated}. That is, we assume that
\begin{align*}
    \overline{F}^k \cap \overline{F}^\ell = \emptyset && \forall k, \ell = 1, \ldots, N_f, k \neq \ell.
\end{align*}
The \emph{defeatured domain} $\defdomain$ is then defined by
\begin{align*}
    \defdomain \coloneqq \domain \cup \interior{\bigcup_{\feat \in \featset} \overline{F}}.
\end{align*}
For a feature $\feat \in \featset$, we define the \emph{defeatured boundary} of the feature $\feat$, $\defbd^\feat \coloneqq \interior{\overline{\partial \domain} \cap \overline{\partial \feat}}$. This is the piece of the boundary that is lost, when the feature $\feat$ is removed from the geometry. In contrast, the \emph{simplified boundary} of the feature is then defined as $\simpbd^\feat \coloneqq \partial \feat \setminus \overline{\defbd^\feat}$. This is the piece of the defeatured boundary $\partial \defdomain$ that is added, when the feature $\feat$ is removed from the geometry; see \cref{fig:preliminaries:feature types}. In particular, we have $\partial \feat = \overline{\defbd^\feat} \cup \overline{\simpbd^\feat}$. The complete defeatured and simplified boundaries are, respectively, defined as
\begin{align*}
    \completedefbd \coloneqq \interior{\bigcup_{\feat \in \featset} \overline{\defbd^\feat}}, & & \text{ and } & & \completesimpbd \coloneqq \interior{ \bigcup_{F \in \featset} \overline{\simpbd^\feat}}.
\end{align*}
Note that for an internal feature $\feat$, the simplified boundary $\simpbd^\feat$ is empty; see \cref{fig:preliminaries:internal feature}. Moreover, note that removing a negative feature may change the topology of the domain. This is admissible, as the analysis relies only on the Lipschitz regularity of the different boundary pieces.

\subsection{Functional spaces and operators}
Let $\Lp{\omega}$ denote the standard Lebesgue spaces of exponent $p \in [1, \infty]$, equipped with the standard norm denoted by $||\cdot||_{\Lp{\omega}}$. For a multi-index $\boldsymbol{\alpha} \in \N_0^d$, we denote by $D^{\boldsymbol{\alpha}}$ the partial derivative operator. We write $\hs{\omega}$ for the standard Sobolev space of order $s \geq 0$, endowed with the standard norm $\norm{\cdot}{s}{\omega}$; see e.g. Leoni\cite{leoni_first_2017}. Writing $|\boldsymbol{\alpha}| \coloneqq \sum_{i = 1}^d \alpha_i$, the latter is defined by
\begin{align}
    \label{eq:definition of fractional sobolev norm}
    \norm{v}{s}{\omega}^2 \coloneqq \norm{v}{\lfloor s \rfloor}{\omega}^2 + \seminorm{v}{\theta}{\omega}^2, \quad \quad
    \norm{v}{\lfloor s \rfloor}{\omega}^2 \coloneqq \sum_{0 \leq |\boldsymbol{\alpha}| \leq \lfloor s \rfloor} \Ltwonorm{D^{\boldsymbol{\alpha}} v}{\omega}^2,
    \\
    \seminorm{v}{\theta}{\omega}^2 \coloneqq \sum_{|\boldsymbol{\alpha}| = \lfloor s \rfloor} \int_\omega \int_\omega \frac{|D^{\boldsymbol{\alpha}}v(x) -  D^{\boldsymbol{\alpha}}v(y)|^2}{|x - y|^{n + 2\theta}} \dd{s}(x) \dd{s}(y),
\end{align}
where $s = \lfloor s \rfloor + \theta \geq 0$ and $\dd{s}$ denotes the $(n-1)$-dimensional Hausdorff measure if $d = n-1$ and the Lebesgue measure if $d = n$. The vector-valued Lebesgue and Sobolev spaces are denoted by $\Ltwovec{\genman} = [\Ltwo{\genman}]^n$ and $\hsvec{\genman} = [\hs{\genman}]^{n}$, respectively. Similarly, the tensor-valued Lebesgue space is denoted by $\Ltwotensor{\genman} = [\Ltwo{\genman}]^{n \times n}$.
In addition, we define the average of a function over $\genman$ by
\begin{align*}
    \overline{v}^{\genman} \coloneqq \frac{1}{|\genman|} \int_{\genman} v(x) \, \dd{s}(x), & & \forall v \in \Ltwo{\genman}.
\end{align*}

For a subset $\genbd$ of the Lipschitz boundary $\partial \domain$ with $|\genbd| > 0$, we write $\traceop_\genbd: \hone{\domain} \to \htrace{\genbd}$ for the unique trace operator \cite{leoni_first_2017}. Furthermore, let 
$\liftop_{\partial\domain}: \htrace{\partial\domain} \to \hone{\domain}$ be 
any linear, continuous right inverse of $\traceop_{\partial\domain}$, 
i.e., satisfying $\traceop_{\partial\domain}(\liftop_{\partial\domain}(\mu)) 
= \mu$ for all $\mu \in \htrace{\partial\domain}$. No assumptions on 
$\liftop_{\partial\domain}$ beyond linearity and continuity with respect 
to the $\htrace{\partial\domain}$ norm are required. The lifting operator need not be associated with any specific differential operator.
In particular, for a function $\mu \in \htrace{\genbd}$, we may define 
the subset
\begin{align*}
    \honebd{\mu}{\genbd}{\domain} \coloneqq \left\{v \in \hone{\domain} 
    : \traceop_\genbd(v) = \mu \right\}.
\end{align*}

Finally, we define the Neumann trace operator as the continuous linear operator $\neumop: \hdiv\domain \to \htracedual{\partial \domain}$, such that for $v \in \hdiv\domain$ and $\mu \in \htrace{\partial\domain}$,
\begin{align}
\label{eq:defnition of neumann operator}
    \langle \neumop(v), \mu \rangle \coloneqq \int_\domain v \cdot \nabla \liftop_{\partial \domain}(\mu) \dd{x} + \int_\domain \Div(v) \liftop_{\partial \domain}(\mu) \dd{x},
\end{align}
where $\langle \cdot, \cdot \rangle$ denotes the duality pairing between $\htrace{\partial \domain}$ and $\htracedual{\partial \domain}$. In particular, for $w \in H^{3/2}(\domain)$ such that $\Delta w \in \Ltwo{\domain}$, it holds that $\neumop(\nabla w) \in L^2(\partial \domain)$\cite{lions_non-homogeneous_1972}.

\subsection{Linear quantities of interest}
\label{subsec:linear quantities of interest}
Let us now abstractly introduce linear QoIs, in terms of which we later want to quantify the defeaturing error. To that end, let $V(\domain)$ and $V_0(\defdomain)$ be Hilbert spaces over the exact domain $\domain$ and defeatured domain $\defdomain$, respectively.
Then, a linear QoI is given by a linear functional $L$ in the dual space $V(\domain)'$.
We denote by $L_0 \in V_0(\defdomain)'$ the defeatured counterpart of $L$, whereby we make the following assumption:
\begin{assumption}
\label{assumption:defeatured qoi}
The defeatured QoI $L_0$ satisfies
\begin{align*}
\label{eq:preliminaries:defeatured qoi condition}
L_0(v_0) = L((v_0)_{|\domain}) && \forall v_0 \in V_0.
\end{align*}
\end{assumption}

\Cref{assumption:defeatured qoi} serves as a simplifying assumption in this paper. It expresses that neither $L$ nor $L_0$ interacts with the feature and its boundary, which holds for any localized QoI, but not for global ones such as the compliance.

%% file: tikz/intro_internal.tex
\begin{tikzpicture}
    \def\length{4}
    \def\width{2}
    \def\radius{0.45}

    \draw[thick] (0, 0) rectangle (\length, \width);

    \node at (\length / 5, \width / 2) {$\domain$};

    \filldraw[pattern=dots, pattern color=gray]  (2 * \length / 3, \width / 2) circle [radius=\radius];

    \node[fill=white] at (2 * \length / 3, \width / 2) {$\feat$};

    \node[anchor=north east] at (2 * \length / 3 - \radius / 1.5, \width / 2 - \radius / 1.5) {$\defbd^\feat$};
    
\end{tikzpicture}

%% file: tikz/intro_boundary.tex
\begin{tikzpicture}
    \def\length{4}
    \def\width{2}
    \def\featLength{1}
    \def\featWidth{0.7}

    \coordinate (A) at (0, 0);
    \coordinate (B) at (\length, 0);
    \coordinate (C) at (\length, \width);
    \coordinate (D) at (0, \width);

    \coordinate (a) at ($(\length / 2 - \featLength /2, \width - \featWidth)$);
    \coordinate (b) at ($(\length / 2 + \featLength /2, \width - \featWidth)$);
    \coordinate (c) at ($(\length / 2 + \featLength /2, \width)$);
    \coordinate (d) at ($(\length / 2 - \featLength /2, \width)$);

    \draw[thick] (A) -- (B) -- (C) -- (c) -- (b) -- (a) -- (d) -- (D) -- cycle;
    
    \node at (\length / 5, \width / 2) {$\domain$};

    \filldraw[pattern=dots, pattern color=gray]  (a) rectangle (c);

    \node[fill=white] at ($(a)!0.5!(c)$) {$\feat$};

    \node[anchor=north] at ($(a)!0.5!(b)$) {$\defbd^\feat$};
    \node[anchor=south] at ($(c)!0.5!(d)$) {$\simpbd^\feat$};
    
\end{tikzpicture}

%% file: contents/03_poisson.tex
\section{Defeaturing Poisson's equation}
\label{sec:poisson}
Poisson's equation provides the simplest setting to develop and illustrate the central ideas of our goal-oriented defeaturing framework. This model problem allows us to introduce the classification of features, formulate the defeaturing problem for geometries with multiple Dirichlet and Neumann features, and establish both energy-norm and goal-oriented error estimators. The arguments presented here recall the energy-norm framework by Buffa et al.\cite{buffa_analysis-aware_2022} and Weder and Buffa\cite{weder_analysis-aware_2025} and extend it to a novel goal-oriented estimator, forming the blueprint for the subsequent treatment of linear elasticity and Stokes flow.

To that end, we decompose the set of features $\featset$ from \cref{subsec:domains boundaries features} into the sets of Dirichlet features $\dirfeatset$ and Neumann features $\neumfeatset$, respectively such that $\featset = \dirfeatset \cup \neumfeatset$. For technical reasons, we must distinguish between three different types within the Dirichlet features: First, \emph{internal} features that are not in contact with any other boundaries, denoted by $\dirintfeatset \subset \dirfeatset$; second, \emph{Dirichlet-Dirichlet} features, that only touch other Dirichlet boundaries, denoted by $\dirdirfeatset \subset \dirfeatset$; and third, \emph{Dirichlet-Neumann} features that may also touch other Neumann boundaries, which we denote by $\dirneumfeatset \subset \dirfeatset$. We refer to Weder and Buffa\cite{weder_analysis-aware_2025} for an in-depth discussion of this distinction.

The two different Dirichlet boundary feature types are illustrated in \cref{fig:poisson:dirichlet feature types}. In contrast, no such distinction is necessary for Neumann features\cite{chanon_adaptive_2022}.
\begin{figure}
\centering
    \begin{subfigure}[B]{0.45\textwidth}
    \centering
    \input{tikz/poisson_dd_feature}
    \caption{Dirichlet-Dirichlet feature.}
    \end{subfigure}
    \begin{subfigure}[B]{0.45\textwidth}
    \centering
    \input{tikz/poisson_dn_feature}
    \caption{Dirichlet-Neumann feature.}
    \label{fig:dirichlet neumann feature}
    \end{subfigure}
\caption{Illustration of the two types of Dirichlet boundary features: (a) Dirichlet-Dirichlet feature and (b) Dirichlet-Neumann feature.}

\label{fig:poisson:dirichlet feature types}
\end{figure}
 Moreover, we can decompose the complete defeatured and simplified boundaries as
\begin{align*}
    \defbd = \interior{\overline{\completedirdefbd} \cup \overline{\completeneumdefbd}}, && \text{ and} && \simpbd = \interior{\overline{\completedirsimpbd} \cup \overline{\completeneumsimpbd}},
\end{align*}
where the Dirichlet pieces are defined by
\begin{align*}
    \completedirdefbd \coloneqq \interior{\bigcup_{\feat \in \dirfeatset}\overline{\dirdefbd^\feat}}, && \text{ and } && \completedirsimpbd \coloneqq \interior{\bigcup_{\feat \in \dirdirfeatset \cup \dirintfeatset} \overline{\dirsimpbd^\feat}},
\end{align*}
while the Neumann pieces are defined by
\begin{align*}
    \completeneumdefbd \coloneqq \interior{\bigcup_{\feat \in \neumfeatset}\overline{\dirdefbd^\feat}}, && \text{ and } && \completeneumsimpbd \coloneqq \interior{\bigcup_{\feat \in \neumfeatset \cup \dirneumfeatset} \overline{\dirsimpbd^\feat}}.
\end{align*}

\begin{remark}
On simplified boundaries of Dirichlet–Neumann features, we intentionally impose Neumann conditions to avoid tracking complex intersections with surrounding boundaries. However, in terms of the analysis, one could also impose Dirichlet data.
\end{remark}

To formulate the problem in the exact domain $\domain$, its boundary $\partial \domain$ is decomposed into the Dirichlet and Neumann parts, i.e. $\partial \domain = \rmbd_D \cup \rmbd_N$, where $\rmbd_D \cap \rmbd_N = \emptyset$. Particularly, we have $\completedirdefbd \subset \rmbd_D$ and $\completeneumdefbd \subset \rmbd_N$. Then, the exact problem reads
\begin{equation}
\label{eq:poisson:exact problem}
\begin{cases}
-\Delta u = f & \text{ in } \domain,\\
u = g & \text{ on } \rmbd_D,\\
\partial_n u = h  & \text{ on } \rmbd_N,
\end{cases}
\end{equation}
for boundary data $g \in \htrace{\rmbd_D}$ and $h \in \Ltwo{\rmbd_N}$ and a source term $f \in \Ltwo{\domain}$. Problem \cref{eq:poisson:exact problem} is well-posed in $V(\domain) = \honebd{g}{\rmbd_D}{\domain}$, even with slightly weaker regularity assumptions, but we assume the latter for simplicity\cite{ern_finite_2021}.

Similarly, we can decompose the boundary of the defeatured domain as $\partial \defdomain = \dirrmbd \cup \neumrmbd$. Let $g_0 \in \htrace{\dirrmbd}$ and $h_0 \in \Ltwo{\neumrmbd}$ be extensions of $g$ and $h$, respectively. Then, the defeatured problem reads
\begin{align}
\label{eq:poisson:defeatured problem}
\begin{cases}
    - \Delta u_0 = f & \text{ in } \defdomain,\\
    u_0 = g_0 & \text{ on } \dirrmbd,\\
    \partial_n u_0 = h_0 & \text{ on } \neumrmbd,
\end{cases}
\end{align}
where we write $f$ for an $L^2$-extension of the source term in $\domain$ to $\defdomain$ by abuse of notation.
Problem \cref{eq:poisson:defeatured problem} is also well-posed in $V_0(\defdomain) = \honebd{g_{0}}{\dirrmbd}{\defdomain}$. Having set up the defeaturing problem for Poisson’s equation, we now turn to quantifying the error between $u$ and $u_0$ in the corresponding energy norm.

\subsection{Energy-norm estimates}
In this subsection, we establish energy-norm estimates that quantify the error introduced by defeaturing in the Poisson problem. For ease of notation, we define the defeaturing error function $e \coloneqq u - (u_0)_{|\domain}$, which satisfies the PDE
\begin{align}
\label{eq:poisson:error pde}
\begin{cases}
    -\Delta e = 0 & \text{ in } \domain,\\
    e = 0 & \text{ on } \rmbd_D \setminus \overline{\completedirdefbd},\\
    e = d_{\completedirdefbd} & \text{ on } \completedirdefbd,\\
    \partial_n e = 0 & \text{ on } \rmbd_N \setminus \overline{\completeneumdefbd},\\
    \partial_n e = d_{\completeneumdefbd} & \text{ on } \completeneumdefbd,
\end{cases}
\end{align}
where $d_{\completedirdefbd} \coloneqq g_{|\completedirdefbd} - (u_0)_{|\completedirdefbd}$ and $d_{\completeneumdefbd} \coloneqq h_{|\completeneumdefbd} - (\partial_n u_0)_{|\completeneumdefbd}$. For the individual features, we similarly define the boundary errors by
\begin{align*}
    d_{\defbd^\feat} \coloneqq \begin{cases}
        g_{|\defbd^\feat} - (u_0)_{|\defbd^\feat} & \feat \in \dirfeatset, \\
        h_{|\defbd^\feat} - (\partial_n u_0)_{|\defbd^\feat} & \feat \in \neumfeatset.
    \end{cases}
\end{align*}
Note that in general, the defeaturing error $e$ belongs to $\hone{\domain}$. Throughout, we require $d_{\defbd^\feat} \in \hone{\defbd^\feat}$ for every Dirichlet feature $\feat \in \dirfeatset$, so that the tangential gradient $\nabla_t d_{\defbd^\feat} \in \Ltwo{\defbd^\feat}$ and the estimators below are well-defined. This is essential for computability and is guaranteed, e.g., by $g \in \hone{\completedirdefbd}$ together with $H^{3/2+\delta}$-regularity of $u_0$ near $\completedirdefbd$. The latter holds for internal features by interior elliptic regularity and for boundary features under standard assumptions\cite{weder_analysis-aware_2025}.

For ease of exposition, we further assume $e \in H^{3/2+\delta}(\domain)$, which holds in Lipschitz domains under mild data and boundary regularity\cite{grisvard_elliptic_2011}, so that $\neumop(\nabla e) \in \Ltwo{\partial \domain}$ and simplifies the treatment of Dirichlet-Neumann features. Unlike the requirement above, this is not essential: the analysis extends to the case $e \in \hone{\domain}$ using local extensions\cite{weder_analysis-aware_2025}.

Recall that the energy norm associated with the Poisson problem is given by the $H^1$-seminorm, i.e., $\energynorm{v}{\domain} \coloneqq \honeseminorm{v}{\domain}$ for $v \in \honebd{0}{\rmbd_D}{\domain}$.
Using integration by parts and \cref{eq:poisson:error pde}, we find similarly to Buffa et al.\cite{buffa_analysis-aware_2022} and Weder and Buffa\cite{weder_analysis-aware_2025} that
\begin{nalign}
\label{eq:poisson:error representation}
\honeseminorm{e}{\domain}^2 &= \int_{\partial \domain} \partial_n e \, e \dd{s}
\\
&= \sum_{\feat \in \dirdirfeatset} \int_{\defbd^\feat} \neumop(\nabla e) d_{\dirdefbd^F} \dd{s} + \sum_{\feat \in \dirintfeatset} \int_{\defbd^\feat}\neumop(\nabla e) d_{\dirdefbd^F} \dd{s}
\\
&+ \sum_{\feat \in \dirneumfeatset} \int_{\defbd^\feat} \neumop(\nabla e) d_{\dirdefbd^\feat} \dd{s}
+
\sum_{\feat \in \neumfeatset} \int_{\defbd^\feat} d_{\neumdefbd^\feat} e_{|\neumdefbd^\feat} \dd{s}.
\end{nalign}
To derive an estimator expression from \cref{eq:poisson:error representation}, we follow the arguments of Buffa et al.\cite{buffa_analysis-aware_2022} and Weder and Buffa\cite{weder_analysis-aware_2025}. That is, for the integrals in the four sums of \cref{eq:poisson:error representation}, we seek estimates of the form
\begin{align*}
    \int_{\defbd^\feat} \neumop(\nabla e) d_{\dirdefbd^F} \dd{s} \lesssim \mathcal{E}^{\defbd^\feat}(u_0) \honeseminorm{e}{\domain}, && \text{ or } && \int_{\defbd^\feat} d_{\neumdefbd^\feat} e_{|\neumdefbd^\feat} \dd{s} \lesssim \mathcal{E}^{\defbd^\feat}(u_0) \honeseminorm{e}{\domain},
\end{align*}
for Dirichlet features $\feat \in \dirfeatset$ and Neumann features $\feat \in \neumfeatset$, respectively. Here, $\mathcal{E}^{\defbd^\feat}(u_0)$ denotes a generic estimator expression, whose structure depends on the specific feature type.

Indeed, applying the Cauchy-Schwarz inequality separates the unknown error terms $ \neumop(\nabla e)$ or $e_{|\neumdefbd^\feat}$ from the boundary errors $d_{\defbd^\feat}$. The subsequent bounds rely on Poincaré-type inequalities on the boundary $\defbd^\feat$, the choice of which depends on the behavior of $d_{\defbd^\feat}$ at the intersection with the remaining boundary: For a Dirichlet-Dirichlet feature, the boundary error $d_{\defbd^\feat}$ vanishes at the boundary of the feature and the classical Poincaré inequality applies (\cref{lemma:poincare I}). For a Dirichlet-Neumann feature, the boundary error does not vanish at the boundary of the feature anymore. Hence, the average must be removed before a Poincaré inequality can be applied (\cref{lemma:poincare II}). Similar issues arise for internal Dirichlet and Neumann features. For the details, we refer to Buffa et al.\cite{buffa_analysis-aware_2022} and Weder and Buffa\cite{weder_analysis-aware_2025}.

The following estimator expressions are shown to be reliable for the different feature types in Refs.~\citenum{buffa_analysis-aware_2022} and \citenum{weder_analysis-aware_2025}, respectively:
\begin{align*}
\defestn{u_0}{\neumdefbd^\feat} &\coloneqq \left(|\neumdefbd^\feat|^{\frac{1}{n-1}} \Ltwonorm{d_{\dirdefbd^\feat} - \avg{d_{\dirdefbd^\feat}}{\dirdefbd^\feat}}{\dirdefbd^\feat}^2 + c_{\neumdefbd^\feat}^2 |\neumdefbd^\feat|^{\frac{n}{n-1}} |\avg{d_{\dirdefbd^\feat}}{\dirdefbd^\feat}|^2 \right)^{\frac{1}{2}},\\
\defestdd{u_0}{\dirdefbd^\feat} &\coloneqq \sqrt{2 \Ltwonorm{d_{\dirdefbd^\feat}}{\dirdefbd^\feat} \Ltwonorm{\nabla_t d_{\dirdefbd^\feat}}{\dirdefbd^\feat}},
\\
\defestdn{u_0}{\dirdefbd^\feat} &\coloneqq \sqrt{2 \Ltwonorm{d_{\dirdefbd^\feat} - \avg{d_{\dirdefbd^\feat}}{\dirdefbd^\feat}}{\dirdefbd^\feat} \Ltwonorm{\nabla_t d_{\dirdefbd^\feat}}{\dirdefbd^\feat}} + |\dirdefbd^\feat|^{\frac{n - 2}{2(n-1)}}|\avg{d_{\dirdefbd^\feat}}{\dirdefbd^\feat}|,
\\
\defestint{u_0}{\dirdefbd^\feat} &\coloneqq \tilde{c}_{\dirdefbd^\feat} \sqrt{ \Ltwonorm{d_{\dirdefbd^\feat} - \avg{d_{\dirdefbd^\feat}}{\dirdefbd^\feat}}{\dirdefbd^\feat} \Ltwonorm{\nabla_t d_{\dirdefbd^\feat}}{\dirdefbd^\feat}} + \Bar{c}_{\dirdefbd^\feat}|\avg{d_{\dirdefbd^\feat}}{\dirdefbd^\feat}|.
\end{align*}
Here, $\nabla_t d_{\defbd^\feat}$ denotes the tangential gradient along $\defbd^\feat$.
The constant $c_{\defbd^\feat}$ is defined by
\begin{align*}
    c_{\defbd^\feat} \coloneqq \begin{cases}
        \max\left(-\log(|\defbd^\feat|), \eta\right)^{\frac{1}{2}} & \text{ if } n = 2,
        \\
        1 & \text{ if } n = 3,
    \end{cases}
\end{align*}
where $\eta \in \R$ is the unique solution to $\eta = - \log(\eta)$; see \cite[p. 8]{buffa_analysis-aware_2022}. If the feature $\feat$ is internal such that $\rcirc(\feat) < \dist(c_\feat, \partial \defdomain)$, the constant $\Bar{c}_{\defbd^\feat}$ is defined by
\begin{align*}
        \Bar{c}_{\defbd^\feat} \coloneqq
    \begin{cases}
        \sqrt{\frac{2 \pi}{|\log\left(s_{\defbd^\feat}\right)|}}, & n = 2,\\
        \sqrt{\frac{4 \pi \,\rcirc(\feat)}{1 - s_{\defbd^\feat}}}, & n=3
    \end{cases}, && \text{ where } && s_{\defbd^\feat} \coloneqq \frac{\rcirc(\feat)}{\dist(c_\feat, \partial \defdomain)};
\end{align*}
see Lemma 4.3 in Weder and Buffa\cite{weder_analysis-aware_2025}. Finally, the constant $\tilde{c}_{\dirdefbd^\feat}$ appearing in the internal-feature estimator $\defestint{u_0}{\dirdefbd^\feat}$ is defined by
\begin{align*}
    \tilde{c}_{\dirdefbd^\feat} \coloneqq 1 + \left(1 + \frac{2|\dirdefbd^\feat|^{\frac{1}{n-1}}}{\dist(\partial \defdomain, \dirdefbd^\feat)}\right)^{\frac{1}{2}}.
\end{align*}
The explicit numerals appearing in the estimators above record the multiplicities with which the norms appear in the corresponding reliability proofs. We track these multiplicities whenever possible, absorbing into the $\lesssim$-notation only those constants that cannot be tracked, so as to obtain the tightest possible bounds.

\begin{remark}
\label{rem:poisson:internal constants}
The constant $\tilde{c}_{\dirdefbd^\feat}$ captures the theoretical blow-up of the gradient near the outer boundary $\partial \defdomain$. Computing the exact surface-to-surface distance $\dist(\partial \defdomain, \dirdefbd^\feat)$ is prohibitively expensive in typical CAD pipelines. However, as $|\dirdefbd^\feat| \to 0$ with $|\dirdefbd^\feat|^{\frac{1}{n-1}} / \dist(\partial \defdomain, \dirdefbd^\feat) \ll 1$, we have $\tilde{c}_{\dirdefbd^\feat} \to 2$; we therefore set $\tilde{c}_{\dirdefbd^\feat} \approx 2$ in the numerical experiments. In contrast, the constant $\Bar{c}_{\dirdefbd^\feat}$ must vanish as $|\dirdefbd^\feat| \to 0$ in order to recover the correct asymptotic rate, and its evaluation requires only the inexpensive point-projection query $\dist(c_\feat, \partial \defdomain)$ from the circumcenter $c_\feat$ to the outer boundary.
\end{remark}

Adding up the feature-wise contributions, we then find the global estimate
\begin{align*}
\label{eq:poisson:reliability inequality}
    \honeseminorm{e}{\domain}^2 \lesssim \honeseminorm{e}{\domain} \left(
        \sum_{\feat \in \dirdirfeatset} \defestdd{u_0}{\dirdefbd^\feat}^2 
        + \sum_{\feat \in \dirintfeatset} \defestint{u_0}{\dirdefbd^\feat}^2
        \right .
        \\
        \left .
        + \sum_{\feat \in \dirneumfeatset} \defestdn{u_0}{\dirdefbd^\feat}^2 
        + \sum_{\feat \in \neumfeatset} \defestn{u_0}{\neumdefbd^\feat}^2
    \right)^{\frac{1}{2}}.
\end{align*}
Hence, we define the general multi-feature estimator by
\begin{nalign}
\label{eq:poisson:multifeature estimator}
\multidefest{u_0} \coloneqq \left(
        \sum_{\feat \in \dirdirfeatset} \defestdd{u_0}{\dirdefbd^\feat}^2 
        + \sum_{\feat \in \dirintfeatset} \defestint{u_0}{\dirdefbd^\feat}^2
        \right .
        \\
        \left .
        + \sum_{\feat \in \dirneumfeatset} \defestdn{u_0}{\dirdefbd^\feat}^2 
        + \sum_{\feat \in \neumfeatset} \defestn{u_0}{\neumdefbd^\feat}^2
    \right)^{\frac{1}{2}},
\end{nalign}
whose reliability immediately follows from the reliability of the feature-wise estimators.

\subsection{Goal-oriented estimates}
We now turn to the derivation of goal-oriented estimates for the Poisson problem defined above. To this end, let $L \in V(\domain)'$ be a linear QoI as defined in \cref{subsec:linear quantities of interest} with a defeatured counterpart $L_0 \in V_0(\defdomain)'$ satisfying \cref{assumption:defeatured qoi}. Motivated by the dual-weighted residual method \cite{becker_optimal_2001}, we represent $L$ and $L_0$ with the help of influence functions $z$ and $z_0$, solving the following dual problems, respectively:
\begin{align}
\label{eq:poisson:dual problems}
\begin{cases}
-\Delta z = L & \text{ in } \domain,\\
z = 0 & \text{ on } \rmbd_D,\\
\partial_n z = 0  & \text{ on } \rmbd_N,
\end{cases}
&& \text{ and } &&
\begin{cases}
    - \Delta z_0 = L_0 & \text{ in } \defdomain,\\
    z_0 = 0 & \text{ on } \dirrmbd,\\
    \partial_n z_0 = 0 & \text{ on } \neumrmbd.
\end{cases}
\end{align}
For the well-posedness of these problems, we refer to Agranovich\cite{agranovich_sobolev_2015}. In addition, the energy-based estimate from the previous section still holds. The dual defeaturing error $\epsilon \coloneqq z - (z_0)_{|\domain}$ immediately satisfies the estimate $\honeseminorm{\epsilon}{\domain} \lesssim \multidefest{z_0}$.

To derive a goal-oriented estimate, we first note that in virtue of \cref{assumption:defeatured qoi}, we have
\begin{align*}
    L(u) - L_0(u_0) = L(u) - L((u_0)_{|\domain}) = L(e).
\end{align*}
Next, using the error PDE \cref{eq:poisson:error pde} for the primal error $e$ and the exact dual problem, we obtain from two successive integrations by parts that
\begin{align*}
    L(e) = \int_\domain (-\Delta z) \, e \dd{x} = \int_\domain \nabla z \cdot \nabla e \dd{x} - \int_{\partial \domain} \neumop(\nabla z) \, e \dd{s}
    \\
    =  \int_{\partial \domain} \neumop(\nabla e) \, z \dd{s}  -  \int_{\partial \domain} \neumop(\nabla z) \, e \dd{s}.
\end{align*}
Using the homogeneous boundary conditions satisfied by $e$ on $\rmbd_D \setminus \overline{\completedirdefbd}$ and by $\neumop(\nabla e)$ on $\rmbd_N \setminus \overline{\completeneumdefbd}$ (see \cref{eq:poisson:error pde}), as well as the homogeneous boundary conditions satisfied by $z$ on $\rmbd_D$ and by $\neumop(\nabla z)$ on $\rmbd_N$, we find an error representation similar to
\cref{eq:poisson:error representation}:
\begin{nalign}
\label{eq:poisson:qoi error representation}
    L(e) &= \sum_{\feat \in \neumfeatset} \int_{\defbd^\feat} d_{\defbd^\feat} \,  z_{|\defbd^\feat} \dd{s} - \sum_{\feat \in \dirdirfeatset} \int_{\defbd^\feat} \neumop(\nabla z) \,  d_{\defbd^\feat} \dd{s}
    \\
    &- \sum_{\feat \in \dirintfeatset} \int_{\defbd^\feat} \neumop(\nabla z) \,  d_{\defbd^\feat} \dd{s} - \sum_{\feat \in \dirneumfeatset} \int_{\defbd^\feat} \neumop(\nabla z) \, d_{\defbd^\feat} \dd{s}.
\end{nalign}
Note, however, that \cref{eq:poisson:qoi error representation} still involves the exact dual solution $z$. The idea is now to split the right-hand side of \cref{eq:poisson:qoi error representation} into a term that only involves the defeatured solutions $u_0$ and $z_0$ and a term that only involves the error functions $e$ and $\epsilon$. To that purpose, we define the corrector term

\begin{nalign}
\label{eq:poisson:corrector}
R_L^\defbd(u_0, z_0) &\coloneqq \sum_{\feat \in \neumfeatset} \int_{\defbd^\feat} d_{\defbd^\feat} \,  (z_0)_{|\defbd^\feat} \dd{s} - \sum_{\feat \in \dirdirfeatset} \int_{\defbd^\feat} \neumop(\nabla z_0) \, d_{\defbd^\feat} \dd{s}
\\
&- \sum_{\feat \in \dirintfeatset} \int_{\defbd^\feat} \neumop(\nabla z_0) \, d_{\defbd^\feat} \dd{s} - \sum_{\feat \in \dirneumfeatset} \int_{\defbd^\feat} \neumop(\nabla z_0) \, d_{\defbd^\feat} \dd{s}.
\end{nalign}

With this corrector at hand, we can prove the following result:
\begin{theorem}
\label{thm:poisson:reliability of goal-oriented estimate}
Let $\domain$ be a Lipschitz domain with a set of negative isotropic Lipschitz features $\featset$. Let $L \in V(\domain)'$ and $L_0 \in V_0(\defdomain)'$ be linear QoIs such that \cref{assumption:defeatured qoi} is satisfied. Furthermore, let $u$ and $u_0$ be the solutions to the exact problem \cref{eq:poisson:exact problem} and defeatured problem \cref{eq:poisson:defeatured problem}, respectively, and $z_0$ the solution to the defeatured dual problem \cref{eq:poisson:dual problems}. Then, the following estimate holds:
\begin{align*}
    |L(u) - L_0(u_0) - R_L^\defbd(u_0, z_0)| \lesssim \multidefest{u_0}\, \multidefest{z_0}.
\end{align*}
\end{theorem}

\begin{proof}
Using the error representation \cref{eq:poisson:qoi error representation} and the corrector term from \cref{eq:poisson:corrector}, we obtain

\begin{align*}
L(e) - R_L^\defbd(u_0, z_0) &= \sum_{\feat \in \neumfeatset} \int_{\defbd^\feat} d_{\defbd^\feat} \,  \epsilon_{|\defbd^\feat} \dd{s} - \sum_{\feat \in \dirdirfeatset} \int_{\defbd^\feat} \neumop(\nabla \epsilon) \,  d_{\defbd^\feat}  \dd{s}
\\
&- \sum_{\feat \in \dirintfeatset} \int_{\defbd^\feat} \neumop(\nabla \epsilon) \,  d_{\defbd^\feat}  \dd{s} - \sum_{\feat \in \dirneumfeatset} \int_{\defbd^\feat} \neumop(\nabla \epsilon) \,  d_{\defbd^\feat}  \dd{s}.
\end{align*}
Now, observe that the integrals in the sums above have the same structure as the ones in \cref{eq:poisson:error representation}.
Hence, with the same arguments from Buffa et al.\cite{buffa_analysis-aware_2022} and Weder and Buffa\cite{weder_analysis-aware_2025}, we find
\begin{align*}
\label{eq:poisson:qoi reliability}
|L(e) - R_L^\defbd(u_0, z_0)| \lesssim \honeseminorm{\epsilon}{\domain} \left (
    \sum_{\feat \in \neumfeatset} \defestn{u_0}{\neumdefbd^\feat}^2
    +
    \sum_{\feat \in \dirdirfeatset} \defestdd{u_0}{\dirdefbd^\feat}^2 
    \right .
    \\
    \left .
    +
    \sum_{\feat \in \dirintfeatset} \defestint{u_0}{\dirdefbd^\feat}^2
    +
    \sum_{\feat \in \dirneumfeatset} \defestdn{u_0}{\dirdefbd^\feat}^2 
\right)^{\frac{1}{2}}
\lesssim
\multidefest{u_0}\, \multidefest{z_0}.
\end{align*}
\end{proof}

\begin{remark}
    In light of \cref{thm:poisson:reliability of goal-oriented estimate}, we observe that apart from ensuring the reliability of the goal-oriented estimate, the correction term $R_L^\defbd(u_0, z_0)$ provides a first-order correction of the defeatured QoI $L_0(u_0)$. That is, if the product $\multidefest{u_0}\, \multidefest{z_0}$ is small enough, the estimate implies that the approximation
    \begin{align*}
        L(u) \approx L_0(u_0) + R_L^\defbd(u_0, z_0),
    \end{align*}
    is accurate.
    We will refer to the right-hand side as the \emph{corrected QoI}.
\end{remark}

The same structure now naturally extends to the linear elasticity and Stokes' equations.

%% file: tikz/poisson_dd_feature.tex
\begin{tikzpicture}
    \def\length{4}
    \def\width{2}
    \def\featLength{1}
    \def\featWidth{0.7}

    \coordinate (A) at (0, 0);
    \coordinate (B) at (\length, 0);
    \coordinate (C) at (\length, \width);
    \coordinate (D) at (0, \width);

    \coordinate (a) at ($(\length / 2 - \featLength /2, \width - \featWidth)$);
    \coordinate (b) at ($(\length / 2 + \featLength /2, \width - \featWidth)$);
    \coordinate (c) at ($(\length / 2 + \featLength /2, \width)$);
    \coordinate (d) at ($(\length / 2 - \featLength /2, \width)$);

    \draw[thick] (A) -- (B) -- (C) -- (c) -- (b) -- (a) -- (d) -- (D) -- cycle;
    
    \node at (\length / 5, \width / 2) {$\domain$};
    \node[anchor=west] at ($(B)!0.25!(C)$) {$\rmbd_D$};

    \filldraw[pattern=dots, pattern color=gray]  (a) rectangle (c);

    \node[fill=white] at ($(a)!0.5!(c)$) {$\feat$};
    
\end{tikzpicture}

%% file: tikz/poisson_dn_feature.tex
\begin{tikzpicture}
    \def\length{4}
    \def\width{2}
    \def\featLength{1}
    \def\featWidth{0.7}

    \coordinate (A) at (0, 0);
    \coordinate (B) at (\length, 0);
    \coordinate (C) at (\length, \width);
    \coordinate (D) at (0, \width);

    \coordinate (a) at ($(\length / 2 - \featLength /2, \width - \featWidth)$);
    \coordinate (b) at ($(\length / 2 + \featLength /2, \width - \featWidth)$);
    \coordinate (c) at ($(\length / 2 + \featLength /2, \width)$);
    \coordinate (d) at ($(\length / 2 - \featLength /2, \width)$);

    \draw[thick] (A) -- (B) -- (C) -- (c) -- (b) -- (a) -- (d) -- (D) -- cycle;

    \draw[thick, red1] (D) -- (d);
    \draw[thick, red1] (c) -- (C);

    \node[anchor=south, red1] at ($(c)!0.7!(C)$) {$\rmbd_N$};

    
    \def\offset{0.1}
    \def\neumannOffset{0.66}
    \coordinate (ao) at ($(\length / 2 - \featLength /2 - \offset, \width - \featWidth - \offset)$);
    \coordinate (bo) at ($(\length / 2 + \featLength /2 + \offset, \width - \featWidth - \offset)$);
    \coordinate (co) at ($(\length / 2 + \featLength /2 + \offset, \width - \offset)$);
    \coordinate (do) at ($(\length / 2 - \featLength /2 - \offset, \width - \offset)$);

    \coordinate(cno) at ($(\length / 2 + \featLength /2 + \neumannOffset, \width - \offset)$);
    \coordinate (dno) at ($(\length / 2 - \featLength /2 - \neumannOffset, \width - \offset)$);


    
    
    \node at (\length / 5, \width / 2) {$\domain$};
    \node[anchor=west] at ($(B)!0.25!(C)$) {$\rmbd_D$};

    \filldraw[pattern=dots, pattern color=gray]  (a) rectangle (c);

    \node[fill=white] at ($(a)!0.5!(c)$) {$\feat$};

    
\end{tikzpicture}

%% file: contents/04_elasticity.tex
\section{Defeaturing the linear elasticity equation}
\label{sec:elasticity}
In this section, we present the general defeaturing problem for the linear elasticity equations for mixed negative Dirichlet and Neumann features. It closely follows \cref{sec:poisson}, Antolín and Chanon\cite{antolin_analysisaware_2024} and Chanon\cite{chanon_adaptive_2022}, where the case of Neumann features was addressed.

 For a function $\bv: \domain \to \R^n$, we define the \emph{linearized strain rate tensor} and the \emph{Cauchy stress tensor}, respectively, by
\begin{align*}
    \strain(\bv) \coloneqq \frac{1}{2} (\nabla \bv + \nabla \bv^\top), && \text{ and } && \stress(\bv) \coloneqq 2 \mu \strain(\bv) + \lambda (\nabla \cdot \bv) \mathbbm{I}_n,
\end{align*}
where $\mu$ and $\lambda$ denote the Lamé constants. We recall that due to thermodynamic stability, $\mu > 0$ and $\lambda + \frac{2}{3}\mu > 0$.
Let $\bg \in \htracevec{\rmbd_D}$, $\bh \in \Ltwovec{\rmbd_N}$, and $\bff \in \Ltwovec{\domain}$. Then the linear elasticity equations in the exact domain read:
\begin{equation}
\label{eq:linear elasticity:exact problem}
\begin{cases}
    -\nabla \cdot \stress(\bu) = \bff & \text{ in } \domain,\\
    \bu = \bg & \text{ on } \rmbd_D,\\
    \stress(\bu) \bn = \bh & \text{ on } \rmbd_N,
\end{cases}
\end{equation}
where $\bn$ denotes the unitary outward normal of $\partial\domain$. The corresponding weak formulation reads: Find $\bu \in V(\domain)\coloneqq \honevecbd{\domain}{\bg}{\rmbd_D}$, such that for all $\bv \in \honevecbd{\domain}{\boldsymbol{0}}{\rmbd_D}$,
\begin{align}
\label{eq:linear elasticity:weak form}
\int_\domain \stress(\bu) : \strain(\bv) \dd{x} = \int_\domain \bff \cdot \bv \dd{x} + \int_{\rmbd_N} \bh \cdot \bv \dd{s}.
\end{align} 
The energy norm corresponding to problem \cref{eq:linear elasticity:weak form} is defined by 
\begin{align*}
    \energynorm{\bv}{\domain} \coloneqq \left(\int_\domain \stress(\bv) : \strain(\bv) \dd{x}\right)^{\frac 1 2}, & & \forall \bv \in V(\domain).
\end{align*}

For the defeatured domain $\defdomain$, we write $\bff$ for an $L^2$-extension of the source term in \cref{eq:linear elasticity:exact problem} by abuse of notation.
Once we have chosen extensions $\bg_{0} \in \htracevec{\rmbd_{0, D}}$ and $\bh_0 \in \Ltwovec{\rmbd_{0, N}}$ of $\bg$ and $\bh$, respectively, the defeatured elasticity problem reads:
\begin{align}
\label{eq:linear elasticity:defeatured problem}
\begin{cases}
    -\nabla \cdot \stress( \bu_0) = \bff & \text{ in } \defdomain,\\
     \bu_0 = \bg_{0} & \text{ on } \dirrmbd,\\
    \stress(\bu_0) \bn = \bh_0 & \text{ on } \neumrmbd.
\end{cases}
\end{align}
The exact problem \cref{eq:linear elasticity:exact problem} and its defeatured counterpart \cref{eq:linear elasticity:defeatured problem} are well-posed in $V(\domain)$ and $V_0(\defdomain) \coloneqq \honevecbd{\defdomain}{\bg_0}{\dirrmbd}$, respectively.

Similar to the Poisson problem, we define the vector-valued defeaturing error 
$\be \coloneqq \bu - ( \bu_0)_{|\domain}$. 
For each feature $\feat$, we introduce the boundary error as 
\begin{align*}
    \bd_{\defbd^\feat} \coloneqq 
    \begin{cases}
        \bg_{|\defbd^\feat} - (\bu_0)_{|\defbd^\feat}, & \feat \in \dirfeatset, \\[0.5ex]
        \bh_{|\defbd^\feat} - \stress(\bu_0)\bn_{|\defbd^\feat}, & \feat \in \neumfeatset. 
    \end{cases}
\end{align*}
Analogously to the Poisson case in \cref{sec:poisson}, an integration by parts argument yields a general boundary error representation for the elasticity problem:

\begin{nalign}
\label{eq:elasticity:error representation}
\energynorm{\be}{\domain}^2 = \sum_{\feat \in \dirdirfeatset} \int_{\defbd^\feat} (\stress(\be) \bn) \cdot \bd_{\dirdefbd^F}  \dd{s} + \sum_{\feat \in \dirintfeatset} \int_{\defbd^\feat} (\stress(\be) \bn) \cdot \bd_{\dirdefbd^F}  \dd{s}\\
+ \sum_{\feat \in \dirneumfeatset} \int_{\defbd^\feat} (\stress(\be) \bn) \cdot \bd_{\dirdefbd^\feat}  \dd{s} 
+
\sum_{\feat \in \neumfeatset} \int_{\defbd^\feat} \bd_{\neumdefbd^\feat} \cdot \be_{|\neumdefbd^\feat}  \dd{s}.
\end{nalign}

As in the Poisson problem, we require $\boldsymbol{d}_{\defbd^\feat} \in \honevec{\defbd^\feat}$ for every $\feat \in \dirfeatset$ and, for ease of exposition, $\stress(\be)\bn \in \Ltwovec{\partial \domain}$.
The boundary representation \cref{eq:elasticity:error representation} 
naturally suggests the following computable multi-feature estimator:
\begin{nalign}
\label{eq:elasticity:multifeature estimator}
\multidefest{\bu_0} \coloneqq \left(
        \sum_{\feat \in \dirdirfeatset} \defestdd{\bu_0}{\dirdefbd^\feat}^2 
        + \sum_{\feat \in \dirintfeatset} \defestint{\bu_0}{\dirdefbd^\feat}^2
        \right .
        \\
        \left .
        + \sum_{\feat \in \dirneumfeatset} \defestdn{\bu_0}{\dirdefbd^\feat}^2 
        + \sum_{\feat \in \neumfeatset} \defestn{\bu_0}{\neumdefbd^\feat}^2
    \right)^{\frac{1}{2}},
\end{nalign}
with the feature-wise estimators given by
\begin{align}
\label{eq:elasticity:neumann feature estimator}
\defestn{\bu_0}{\neumdefbd^\feat} &\coloneqq \left(|\neumdefbd^\feat|^{\frac{1}{n-1}} \Ltwonorm{\bd_{\dirdefbd^\feat} - \avg{\bd_{\dirdefbd^\feat}}{\dirdefbd^\feat}}{\dirdefbd^\feat}^2 + c_{\neumdefbd^\feat}^2 |\neumdefbd^\feat|^{\frac{n}{n-1}} \vecnorm{\avg{\bd_{\dirdefbd^\feat}}{\dirdefbd^\feat}}^2 \right)^{\frac{1}{2}},
\\
\label{eq:elasticity:dirdir feature estimator}
\defestdd{\bu_0}{\dirdefbd^\feat} &\coloneqq 2\sqrt{ \Ltwonorm{\bd_{\dirdefbd^\feat}}{\dirdefbd^\feat} \Ltwonorm{\nabla_t \bd_{\dirdefbd^\feat}}{\dirdefbd^\feat}},
\\
\label{eq:elasticity:dirneum feature estimator}
\defestdn{\bu_0}{\dirdefbd^\feat} &\coloneqq 2\sqrt{ \Ltwonorm{\bd_{\dirdefbd^\feat} - \avg{\bd_{\dirdefbd^\feat}}{\dirdefbd^\feat}}{\dirdefbd^\feat} \Ltwonorm{\nabla_t \bd_{\dirdefbd^\feat}}{\dirdefbd^\feat}} + 2|\dirdefbd^\feat|^{\frac{n - 2}{2(n-1)}}\vecnorm{\avg{\bd_{\dirdefbd^\feat}}{\dirdefbd^\feat}},
\\
\label{eq:elasticity:dirint feature estimator}
\defestint{\bu_0}{\dirdefbd^\feat} &\coloneqq \tilde{c}_{\dirdefbd^\feat} \sqrt{ \Ltwonorm{\bd_{\dirdefbd^\feat} - \avg{\bd_{\dirdefbd^\feat}}{\dirdefbd^\feat}}{\dirdefbd^\feat} \Ltwonorm{\nabla_t \bd_{\dirdefbd^\feat}}{\dirdefbd^\feat}} + \Bar{c}_{\dirdefbd^\feat}\vecnorm{\avg{\bd_{\dirdefbd^\feat}}{\dirdefbd^\feat}}.
\end{align}
 The reliability of \cref{eq:elasticity:neumann feature estimator} for a Neumann feature was shown by Antolín and Chanon\cite{antolin_analysisaware_2024}. For the proof of reliability of the estimators for Dirichlet features \cref{eq:elasticity:dirdir feature estimator,eq:elasticity:dirneum feature estimator,eq:elasticity:dirint feature estimator}, we refer to \cref{thm:linear elasticity:dirichlet estimator reliability} in \cref{sec:linear elasticity:reliability}. The feature-wise reliability of the estimators thus implies that $\energynorm{\be}{\domain} \lesssim \multidefest{\bu_0}$ like in the case of the Poisson problem. As for the Poisson estimators, the numerals in the estimator definitions record norm multiplicities tracked in the reliability proofs.

 Consider now a linear QoI $L \in V(\domain)'$ and a corresponding defeatured QoI $L_0 \in V_0(\defdomain)'$ satisfying \cref{assumption:defeatured qoi}. Then, similar to the Poisson problem in \cref{sec:poisson}, we define the exact and defeatured dual problems,
\begin{align}
\label{eq:elasticity:dual problems}
\begin{cases}
    -\nabla \cdot \stress( \bz) = L & \text{ in } \domain,\\
     \bz = \boldsymbol{0} & \text{ on } \rmbd_D,\\
    \stress(\bz) \bn = \boldsymbol{0} & \text{ on } \rmbd_N,
\end{cases}
&& \text{ and } &&
\begin{cases}
    -\nabla \cdot \stress( \bz_0) = L_0 & \text{ in } \defdomain,\\
    \bz_0 = \boldsymbol{0} & \text{ on } \dirrmbd,\\
    \stress(\bz_0) \bn = \boldsymbol{0} & \text{ on } \neumrmbd.
\end{cases}
\end{align}
For the dual defeaturing error $\boldsymbol{\epsilon} \coloneqq \bz - (\bz_0)_{|\domain}$, the energy-norm estimate $\energynorm{\boldsymbol{\epsilon}}{\domain} \lesssim \multidefest{\bz_0}$ also holds.

In analogy to the corrector term \cref{eq:poisson:corrector} for the Poisson problem, we define

\begin{align*}
    R_L^\defbd(\bu_0, \bz_0)  &\coloneqq \sum_{\feat \in \neumfeatset} \int_{\defbd^\feat} \bd_{\defbd^\feat} \cdot (\bz_0)_{|\defbd^\feat} \dd{s} - \sum_{\feat \in \dirdirfeatset} \int_{\defbd^\feat} (\stress(\bz_0)\bn) \cdot \bd_{\defbd^\feat} \dd{s}
    \\
    &- \sum_{\feat \in \dirintfeatset} \int_{\defbd^\feat} (\stress(\bz_0)\bn) \cdot \bd_{\defbd^\feat} \dd{s} - \sum_{\feat \in \dirneumfeatset} \int_{\defbd^\feat} (\stress(\bz_0)\bn) \cdot \bd_{\defbd^\feat} \dd{s}.
\end{align*}
Then, the same reasoning as for the Poisson equation yields the following result:
\begin{theorem}
\label{thm:linear elasticity:reliability of goal-oriented estimate}
Let $\domain$ be a Lipschitz domain with a set of negative isotropic Lipschitz features $\featset$. Let $L \in V(\domain)'$ and $L_0 \in V_0(\defdomain)'$ be linear QoIs such that \cref{assumption:defeatured qoi} is satisfied. Furthermore, let $\bu$ and $\bu_0$ be the solutions to the exact problem \cref{eq:linear elasticity:exact problem} and defeatured problem \cref{eq:linear elasticity:defeatured problem}, respectively, and $\bz_0$ the solution to the defeatured dual problem \cref{eq:elasticity:dual problems}. Then, the following estimate holds:
\begin{align*}
    |L(\bu) - L_0(\bu_0) - R_L^\defbd(\bu_0, \bz_0)| \lesssim \multidefest{\bu_0}\, \multidefest{\bz_0}.
\end{align*}
\end{theorem}

%% file: contents/05_stokes.tex
\section{Defeaturing the Stokes equations}
\label{sec:stokes}
This section addresses the general defeaturing problem for the Stokes equations for mixed negative Dirichlet and Neumann features. It closely follows \cref{sec:poisson,sec:elasticity}.

To that end, let $\strain(\bv)$ denote again the linearized strain rate tensor for a function $\bv: \domain \to \R^n$. In contrast to \cref{sec:elasticity}, we denote by $\stress(\bv)$ the viscous stress tensor of the fluid, which for a Newtonian fluid is given by
\begin{align*}
    \stress(\bv) = 2 \mu \strain(\bv) + \lambda (\nabla \cdot \bv) \mathbbm{I}_n,
\end{align*}
with this time the constants $\mu > 0$ and $\lambda \geq 0$ denoting the dynamic and bulk viscosities, respectively. Let $\bg \in \htracevec{\rmbd_D}$, $\bh \in \Ltwovec{\rmbd_N}$, $\bff \in \Ltwovec{\domain}$, and $f_c \in \Ltwo{\domain}$. Then, the Stokes problem in the exact domain is given by
\begin{align}
\label{eq:stokes:exact problem}
\begin{cases}
    -\nabla \cdot \stress(\bu) + \nabla p = \bff & \text{ in } \domain,\\
    \nabla \cdot \bu = f_c & \text{ in } \domain,\\
    \bu = \bg & \text{ on } \rmbd_D,\\
    \stress(\bu) \bn - p \bn = \bh & \text{ on } \rmbd_N,
\end{cases}
\end{align}
where $\bn$ denotes the unitary outward normal of $\partial\domain$. The corresponding mixed weak formulation reads: Find $\bu \in V(\domain) \coloneqq \honevecbd{\domain}{\bg}{\rmbd_D}$ and $p \in Q(\domain) \coloneqq \Ltwo{\domain}$, such that for all $\bv \in \honevecbd{\domain}{\boldsymbol{0}}{\rmbd_D}$ and $q \in \Ltwo{\domain}$,
\begin{align}
\label{eq:stokes:weak form:momentum}
&a_\domain(\bu, \bv) + b_\domain(\bv, p) = \int_\domain \bff \cdot \bv \dd{x} + \int_{\rmbd_N} \bh \cdot \bv \dd{s},
\\
\label{eq:stokes:weak form:pressure}
&b_\domain(\bu, q) = -\int_\domain f_c \, q \dd{x},
\end{align}
where the bilinear forms $a_\domain: \honevec{\domain} \times \honevec{\domain} \to \R$ and $b_\domain: \honevec{\domain}\times \Ltwo{\domain} \to \R$ are given by
\begin{align*}
    a_\domain(\bu, \bv) \coloneqq \int_\domain \stress(\bu):\strain(\bv) \dd{x}, && \text{ and } && b_\domain(\bv, q) \coloneqq -\int_\domain \nabla \cdot \bv q \dd{x}.
\end{align*}
This problem is well-posed in $V(\domain) \times Q(\domain)$ provided that $\rmbd_N \neq \emptyset$; see Boffi et al.\cite{boffi_mixed_2013}. Otherwise, the pressure space must be restricted to zero-average fields $L^2_0(\Omega)$.
The energy norm associated with the weak formulation \cref{eq:stokes:weak form:momentum,eq:stokes:weak form:pressure} is given by
\begin{align*}
    \energynorm{(\bv, q)}{\domain} \coloneqq a_\domain(\bv, \bv)^{\frac{1}{2}} + \Ltwonorm{q}{\domain}, & & \forall \bv \in \honevec{\domain}, q \in \Ltwo{\domain}.
\end{align*}

For the defeatured domain $\defdomain$, we write $\bff$ and $f_c$ for $L^2$-extensions of the original source terms in \cref{eq:stokes:exact problem} by abuse of notation. Once we have chosen extensions $\bg_{0} \in \htracevec{\dirrmbd}$ and $\bh_0 \in \Ltwo{\neumrmbd}$ of $\bg$ and $\bh$, respectively, the defeatured Stokes problem reads:
\begin{align}
\label{eq:stokes:defeatured problem}
\begin{cases}
    -\nabla \cdot \stress(\bu_0) + \nabla p_0 = \bff & \text{ in } \defdomain,\\
    \nabla \cdot \bu_0 = f_c & \text{ in } \defdomain,\\
    \bu_0 = \bg_0 & \text{ on } \dirrmbd,\\
    \stress(\bu_0) \bn - p_0 \bn = \bh_0 & \text{ on } \neumrmbd.
\end{cases}
\end{align}
This problem is well-posed in $V_0(\defdomain) \times Q_0(\defdomain)$ with $V_0(\defdomain) \coloneqq \honevecbd{\defdomain}{\bg_0}{\dirrmbd}$ and $Q_0(\defdomain) \coloneqq \Ltwo{\defdomain}$, whenever $\rmbd_N \neq \emptyset$, or with $Q_0(\defdomain) = L^2_0(\defdomain)$ otherwise.

The defeaturing error functions in the velocity and pressure variables are defined by $\be_u \coloneqq \bu - (\bu_0)_{|\domain}$ and $e_p \coloneqq p - (p_0)_{|\domain}$, respectively.
Moreover, we define the boundary error for each feature by
\begin{align*}
    \bd_{\defbd^\feat} \coloneqq 
    \begin{cases}
        \bg_{|\defbd^\feat} - (\bu_0)_{|\defbd^\feat}, & \feat \in \dirfeatset, \\[0.5ex]
        \bh_{|\defbd^\feat} - (\stress(\bu_0)\bn_{|\defbd^\feat} - p_0 \bn_{|\defbd^\feat}), & \feat \in \neumfeatset. 
    \end{cases}
\end{align*}
As in the Poisson and elasticity cases, an integration by parts argument shows that the defeaturing error can be expressed entirely through boundary terms:

\begin{nalign}
\label{eq:stokes:error representation}
\energynorm{(\be_u, e_p)}{\domain}^2 &= \sum_{\feat \in \dirdirfeatset} \int_{\defbd^\feat} (\stress(\be_u) \bn - e_p \bn)\cdot \bd_{\dirdefbd^F} \dd{s}
\\
&+ \sum_{\feat \in \dirintfeatset} \int_{\defbd^\feat} (\stress(\be_u) \bn - e_p \bn) \cdot \bd_{\dirdefbd^F} \dd{s}
\\
&+ \sum_{\feat \in \dirneumfeatset} \int_{\defbd^\feat} (\stress(\be_u) \bn - e_p \bn) \cdot \bd_{\dirdefbd^\feat} \dd{s} + \sum_{\feat \in \neumfeatset} \int_{\defbd^\feat} \bd_{\neumdefbd^\feat} \cdot \be_{|\neumdefbd^\feat} \dd{s}.
\end{nalign}
As in the Poisson and elasticity problems, we require $\boldsymbol{d}_{\defbd^\feat} \in \honevec{\defbd^\feat}$ for every $\feat \in \dirfeatset$ and, for ease of exposition, $\stress(\be_u) \bn - e_p \bn \in \Ltwovec{\partial \domain}$.

In direct analogy with the Poisson and elasticity cases, the boundary 
representation motivates the following multi-feature estimator:
\begin{nalign}
\label{eq:stokes:multifeature estimator}
\multidefest{\bu_0, p_0} \coloneqq \left(
        \sum_{\feat \in \dirdirfeatset} \defestdd{\bu_0}{\dirdefbd^\feat}^2
        + \sum_{\feat \in \dirintfeatset} \defestint{\bu_0}{\dirdefbd^\feat}^2
        \right .
        \\
        \left .
        + \sum_{\feat \in \dirneumfeatset} \defestdn{\bu_0}{\dirdefbd^\feat}^2 
        + \sum_{\feat \in \neumfeatset} \defestn{\bu_0, p_0}{\neumdefbd^\feat}^2
    \right)^{\frac{1}{2}},
\end{nalign}
with the feature-wise estimators%
\begin{align}
\label{eq:stokes:neumann feature estimator}
\defestn{\bu_0, p_0}{\neumdefbd^\feat} &\coloneqq \left(|\neumdefbd^\feat|^{\frac{1}{n-1}} \Ltwonorm{\bd_{\dirdefbd^\feat} - \avg{\bd_{\dirdefbd^\feat}}{\dirdefbd^\feat}}{\dirdefbd^\feat}^2 + c_{\neumdefbd^\feat}^2 |\neumdefbd^\feat|^{\frac{n}{n-1}} \vecnorm{\avg{\bd_{\dirdefbd^\feat}}{\dirdefbd^\feat}}^2 \right)^{\frac{1}{2}},
\\
\label{eq:stokes:dirdir feature estimator}
\defestdd{\bu_0}{\dirdefbd^\feat} &\coloneqq 8\sqrt{ \Ltwonorm{\bd_{\dirdefbd^\feat}}{\dirdefbd^\feat} \Ltwonorm{\nabla_t \bd_{\dirdefbd^\feat}}{\dirdefbd^\feat}},
\\
\label{eq:stokes:dirneum feature estimator}
\defestdn{\bu_0}{\dirdefbd^\feat} &\coloneqq 8\sqrt{ \Ltwonorm{\bd_{\dirdefbd^\feat} - \avg{\bd_{\dirdefbd^\feat}}{\dirdefbd^\feat}}{\dirdefbd^\feat} \Ltwonorm{\nabla_t \bd_{\dirdefbd^\feat}}{\dirdefbd^\feat}} + 8|\dirdefbd^\feat|^{\frac{n - 2}{2(n-1)}}\vecnorm{\avg{\bd_{\dirdefbd^\feat}}{\dirdefbd^\feat}},
\\
\label{eq:stokes:dirint feature estimator}
\defestint{\bu_0}{\dirdefbd^\feat} &\coloneqq 4\tilde{c}_{\dirdefbd^\feat} \sqrt{ \Ltwonorm{\bd_{\dirdefbd^\feat} - \avg{\bd_{\dirdefbd^\feat}}{\dirdefbd^\feat}}{\dirdefbd^\feat} \Ltwonorm{\nabla_t \bd_{\dirdefbd^\feat}}{\dirdefbd^\feat}} + 4\Bar{c}_{\dirdefbd^\feat}\vecnorm{\avg{\bd_{\dirdefbd^\feat}}{\dirdefbd^\feat}}.
\end{align}
The reliability of \cref{eq:stokes:neumann feature estimator} for Neumann features was shown by Antolín and Chanon\cite{antolin_analysisaware_2024}.
For the proof of reliability of the estimators for Dirichlet features \cref{eq:stokes:dirdir feature estimator,eq:stokes:dirneum feature estimator,eq:stokes:dirint feature estimator}, we refer to \cref{thm:stokes:dirichlet estimator reliability} in \cref{sec:stokes:reliability}. The reliability for the estimator expressions \cref{eq:stokes:neumann feature estimator,eq:stokes:dirdir feature estimator,eq:stokes:dirneum feature estimator,eq:stokes:dirint feature estimator} immediately implies the reliability of the multi-feature estimator \cref{eq:stokes:multifeature estimator}, i.e. we have $\energynorm{(\be_u, e_p)}{\domain} \lesssim \multidefest{\bu_0, p_0}$. As in the Poisson and linear elasticity problems, the numerals in the estimator definitions stem from tracked norm multiplicities in the reliability proofs, and are not absorbed into the hidden constant.

We consider now a linear QoI $L \in V(\domain)'$ for the velocity variable with a defeatured counterpart $L_0 \in V_0(\defdomain)'$ such that \cref{assumption:defeatured qoi} is satisfied. Then, the exact and defeatured dual problems, respectively, read
\begin{align}
\label{eq:stokes:exact dual problem}
\begin{cases}
    -\nabla \cdot \stress(\bz) + \nabla \zeta = L & \text{ in } \domain,\\
    \nabla \cdot \bz = 0 & \text{ in } \domain,\\
     \bz = \boldsymbol{0} & \text{ on } \rmbd_D,\\
    \stress(\bz) \bn - \zeta \bn = \boldsymbol{0} & \text{ on } \rmbd_N,
\end{cases}
\end{align}
and
\begin{align}
\label{eq:stokes:defeatured dual problem}
\begin{cases}
    -\nabla \cdot \stress(\bz_0) + \nabla \zeta_0 = L_0 & \text{ in } \defdomain,\\
    \nabla \cdot \bz_0 = 0 & \text{ in } \defdomain,\\
    \bz_0 = \boldsymbol{0} & \text{ on } \dirrmbd,\\
    \stress(\bz_0) \bn - \zeta_0 \bn = \boldsymbol{0} & \text{ on } \neumrmbd.
\end{cases}
\end{align}
For Stokes' equations, we define the corrector term
\begin{align*}
    R_L^\defbd(\bu_0, p_0, \bz_0, \zeta_0)  &\coloneqq \sum_{\feat \in \neumfeatset} \int_{\defbd^\feat} \bd_{\defbd^\feat} \cdot (\bz_0)_{|\defbd^\feat} \dd{s}
    \\
    &- \sum_{\feat \in \dirdirfeatset} \int_{\defbd^\feat} (\stress(\bz_0) \bn - \zeta_0 \bn) \cdot \bd_{\defbd^\feat} \dd{s}
    \\
    &- \sum_{\feat \in \dirintfeatset} \int_{\defbd^\feat} (\stress(\bz_0) \bn - \zeta_0 \bn) \cdot \bd_{\defbd^\feat} \dd{s}
    \\
    &- \sum_{\feat \in \dirneumfeatset} \int_{\defbd^\feat} (\stress(\bz_0) \bn - \zeta_0 \bn) \cdot \bd_{\defbd^\feat} \dd{s}.
\end{align*}

Finally, we can repeat the same argument as for the Poisson and linear elasticity problems in \cref{sec:poisson,sec:elasticity} to obtain the following result:
\begin{theorem}
\label{thm:stokes:reliability of goal-oriented estimate}
Let $\domain$ be a Lipschitz domain with negative isotropic Lipschitz features $\featset$. Let $L \in V(\domain)'$ and $L_0 \in V_0(\defdomain)'$ be linear QoIs such that \cref{assumption:defeatured qoi} is satisfied. Furthermore, let $(\bu, p)$ and $(\bu_0, p_0)$ be the solutions to the exact problem \cref{eq:stokes:exact problem} and defeatured problem \cref{eq:stokes:defeatured problem}, respectively, and $(\bz_0, \zeta_0)$ the solution to the defeatured dual problem \cref{eq:stokes:defeatured dual problem}. Then, the following estimate holds:
\begin{align*}
    |L(\bu) - L_0(\bu_0) - R_L^\defbd(\bu_0, p_0, \bz_0, \zeta_0)| \lesssim \multidefest{\bu_0, p_0}\, \multidefest{\bz_0, \zeta_0}.
\end{align*}
\end{theorem}

\begin{remark}
    A similar estimate can be obtained by the same argument for a linear QoI satisfying \cref{assumption:defeatured qoi} in the pressure variable. We report the definitions of the dual problem and the corrector term in \cref{sec:stokes:pressure qoi}. The proof is omitted for brevity.
\end{remark}

%% file: contents/06_numerical_experiments.tex
\section{Numerical experiments}
\label{sec:numerical experiments}
We now validate the theoretical framework through numerical experiments for the Poisson, linear elasticity, and Stokes problems, assessing the reliability and performance of the goal-oriented estimators. The computational geometries and meshes were created with \texttt{gmsh}\cite{geuzaine_gmsh_2009}. Crucially, \texttt{gmsh} generates meshes that are conformal across lower-dimensional entities, allowing for the simultaneous meshing of the defeatured and exact geometries. In particular, comparing the defeatured to the exact solution is straightforward.

The finite element simulations were implemented with the help of the \texttt{FEniCSx} library\cite{baratta_dolfinx_2023}. 
Lagrange elements of degree three were used for the Poisson and linear elasticity problems, while Taylor-Hood elements were used for the Stokes problem to guarantee stability. The meshes were locally refined around the features, with a refinement factor proportional to the feature size and a fixed-size transition layer between the refined and unrefined regions. This refinement targets the singularity of the exact solution $u$ near the feature. It is not required by the estimators themselves, since these depend on the defeatured solution $u_0$. The latter is not affected by the singularity created by the feature. Therefore, $u_0$ and the estimators can easily be resolved on a coarse mesh. The refinement was thus introduced solely to obtain fully converged reference errors for comparison, rather than to improve the estimators' convergence.

The \emph{effectivity index} $\esteff \coloneqq \multidefest{u_0} / \energynorm{u}{\domain}$
will be used to analyze the performance of the estimators. Indeed, according to the theory, this quantity should be asymptotically independent of the feature size in the limit $|\defbd| \to 0$. Furthermore, it captures geometric properties that were neglected in the analysis, such as the curvature of the feature boundary.

\subsection{Poisson's equation}

\begin{figure}
    \centering
    \begin{subfigure}[B]{0.45\textwidth}
    \centering
    \input{tikz/poisson_l2}
    \caption{}
    \label{fig:poisson:l2 illustration}
    \end{subfigure}
    \hspace{5mm}
    \begin{subfigure}[B]{0.45\textwidth}
    \centering
    \input{tikz/poisson_green}
    \caption{}
    \label{fig:poisson:green illustration}
    \end{subfigure}
    \caption{Illustration of the geometries for the Poisson experiments. The dashed square $S$ represents the integration region for the QoI, while the dotted areas $\feat$ represent the negative feature cut out from the defeatured domain.}
    \label{fig:poisson:illustrations}
\end{figure}
We start by considering two cases for the Poisson equation in different domains illustrated in \cref{fig:poisson:illustrations}: One in the unit square $\defdomain = [-\tfrac{1}{2}, \tfrac{1}{2}]^2$ with a feature cut out from the boundary (\cref{fig:poisson:l2 illustration}) and another one in the same unit square with an internal feature cut out at the center (\cref{fig:poisson:green illustration}). In both experiments, we prescribe homogeneous Dirichlet boundary conditions on all defeatured boundaries and the source term is given by
\begin{align*}
    f(\bx) = \frac{A}{\sqrt{2 \pi \sigma^2}} \exp\left(-\frac{1}{2 \sigma^2}\vecnorm{\bx}^2\right),
\end{align*}
with $A = 10$ and $\sigma^2 = 0.01$. We extend the source term to $\defdomain$ by the same expression for the defeatured problems.
Furthermore, the QoI is given by the expression
\begin{align*}
    L(v) = \int_S v \dd{\bx},
\end{align*}
where the region $S$ varies between experiments.

\begin{figure}
    \centering
    \begin{subfigure}[T]{0.45\textwidth}
    \includegraphics{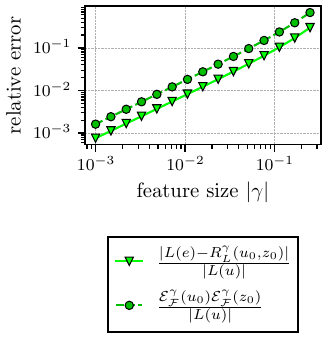}
    \caption{}
    \label{fig:poisson:numerical results:qoi error}
    \end{subfigure}
    \hfill
    \begin{subfigure}[T]{0.45\textwidth}
        \includegraphics{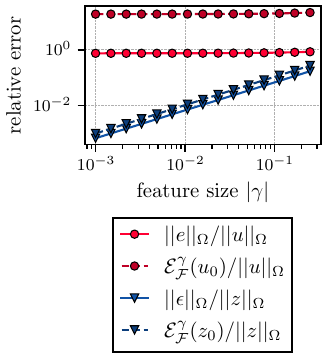}
    \caption{}
    \label{fig:poisson:numerical results:primal dual error}
    \end{subfigure}
    \begin{subfigure}[T]{0.45\textwidth}
        \includegraphics{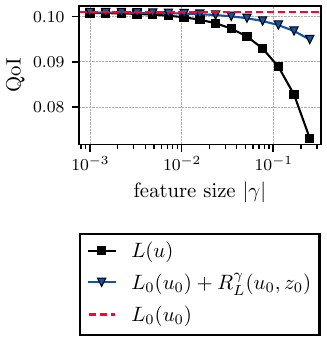}
    \caption{}
    \label{fig:poisson:numerical results:qoi values}
    \end{subfigure}
    \hfill
    \begin{subfigure}[T]{0.45\textwidth}
        \includegraphics{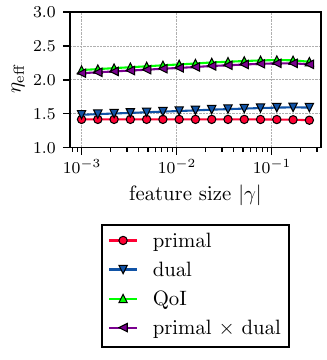}
    \caption{}
    \label{fig:poisson:numerical results:eff}
    \end{subfigure}
    \caption{Results for the Poisson experiment with boundary feature: (a) relative QoI errors and estimates, (b) relative primal and dual defeaturing errors and estimates, (c) QoI values, (d) effectivity indices.}
    \label{fig:poisson:numerical results}
\end{figure}

\subsubsection{A feature on the boundary}
In this first experiment, we consider a feature on the boundary, for which we must yet prescribe a Dirichlet boundary condition. We impose the condition
\begin{align*}
g(\theta) = \sin(\theta), & & \theta = \arctan\left(\frac{x_2 - \frac{1}{2}}{x_1}\right), & & \bx = [x_1, x_2]^\top \in \defbd,
\end{align*}
which is compatible with the homogeneous Dirichlet boundary conditions on the neighboring boundaries. This boundary condition will introduce a strong singularity in the solution $u$ in the exact domain $\domain$.
In this example, the area of influence $S$ is given by
\begin{align*}
S = \left\{\bx \in \R^2 | |x_1| \leq \tfrac{1}{4}, |x_2| \leq \tfrac{1}{4}\right\}.
\end{align*}

\Cref{fig:poisson:numerical results} plots the results for the QoI and the different defeaturing errors against the feature size $|\defbd| \in [10^{-3}, \tfrac{1}{4}]$. More precisely, \cref{fig:poisson:numerical results:qoi values,fig:poisson:numerical results:qoi error} show that the estimate for the QoI error is reliable and that the relative error in the QoI vanishes rapidly in this experiment for $|\defbd| \to 0$. In contrast, while \cref{fig:poisson:numerical results:primal dual error} shows that the estimates in the primal and dual variables are reliable, the relative defeaturing error in the primal variable remains approximately constant at 100\%. The latter is due to the singularity introduced by our specific choice of boundary conditions.

In \cref{fig:poisson:numerical results:eff}, the effectivity indices for the primal and dual variables, as well as the corrected QoI, are plotted. They are indeed asymptotically constant in the limit $|\defbd| \to 0$. The slight drop at the lower end can be attributed to the numerical error, as resolving such small features is challenging and requires strong local mesh refinement, especially for the non-smooth dual problem. Furthermore, the effectivity index of the QoI estimate is indeed the product of the effectivity index of the primal and dual variables.

This experiment demonstrates that in general the energy-norm error is a poor proxy for the error in the QoI, underscoring the need for a goal-oriented estimate.

\subsubsection{An internal feature}
In this experiment, we also impose homogeneous Dirichlet boundary conditions on the feature boundary $\defbd$, such that no singularity is introduced at the feature boundary, in contrast to the first example.
The area of influence $S$ is given by
\begin{align*}
S = \{\bx \in \R^2 \mid |x_1| \leq \tfrac{1}{4}, -0.4 \leq x_2 \leq - 0.15\}.
\end{align*}

For an inclusion subject to Dirichlet boundary conditions in Poisson problems, it is well-known from singular perturbation theory that the convergence of the defeatured solution $u_0$ to $u$ is extremely slow if the source term is extended naively into the inclusion\cite{mazya_asymptotic_2000}, i.e. we have $\energynorm{u - (u_0)_{|\domain}}{\domain} \simeq \mathcal{O}(|\log |\defbd||^{-1/2})$ for $\domain \subset \R^2$. Naturally, the same is true for the dual solutions $z$ and $z_0$. Hence, we expect both substantial discrepancies in the energy norm of the primal and dual variables as well as the difference between the exact QoI $L(u)$ and the corrected defeatured QoI $L_0(u_0) + R_L^\defbd(u_0, z_0)$ in this case.

However, singular perturbation theory also provides first-order approximations to $u$ and $z$ as follows: Let $\hat{G}(x) = \frac{1}{2\pi}\log(|x - m_\feat|)$ denote the fundamental solution of the Laplacian centered at the feature's barycenter $m_\feat$ and consider $G(x) = \hat{G}(x) + g(x)$, where $g$ is the solution to
\begin{align}
\label{eq:num:green}
    \begin{cases}
        -\Delta g = 0,  \text{ in } \defdomain,\\
        g(x) = -\hat{G}(x)  \text{ on } \dirrmbd, \\
        \partial_n g(x) = -\partial_n\hat{G}(x)  \text{ on } \neumrmbd.
    \end{cases}
\end{align}
If the feature $\feat$ is a disk centered at $m_\feat$, the so-called \emph{gauge function} is given by
\begin{align*}
    \mu(\defbd) \coloneqq \frac{2\pi}{\log(\diam\defbd/2) - 2\pi \avg{g}{\defbd}},
\end{align*}
while for more general shapes, $\mu(\defbd)$ additionally depends on the logarithmic capacity of the feature.\cite{mazya_asymptotic_2000,ransford_computation_2011}
Then, the first-order approximations are given by
\begin{align*}
    u_1(x) &\coloneqq u_0(x) + \mu(\defbd) \avg{\bderr}{\defbd} G(x),
    \\
    z_1(x) &\coloneqq z_0(x) - \mu(\defbd) \avg{z_0}{\defbd} G(x),
\end{align*}
such that $\energynorm{u - u_1}{\domain} \simeq \energynorm{z - z_1}{\domain} \simeq \mathcal{O}(|\defbd|)$ for $|\defbd| \to 0$\cite{mazya_asymptotic_2000}.

\begin{figure}
    \centering
    \begin{subfigure}[T]{\linewidth}
    \centering
    \includegraphics[width=\linewidth]{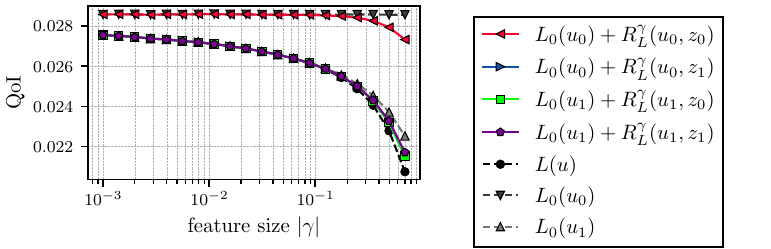}
    \caption{}
    \label{fig:poisson:green:qoi values}
    \end{subfigure}
    \begin{subfigure}[T]{\linewidth}
    \centering
    \includegraphics[width=\linewidth]{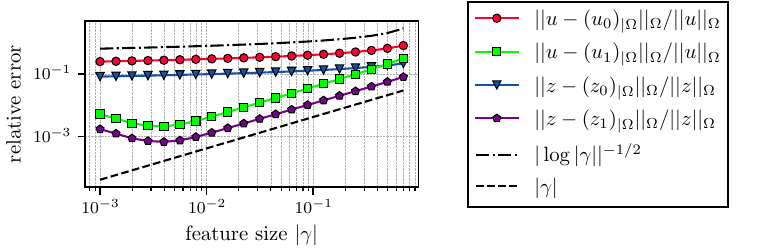}
    \caption{}
    \label{fig:poisson:green:primal dual errors}
    \end{subfigure}
    \caption{Results for the Poisson experiment with internal feature: (a) QoI values and (b) relative defeaturing errors for different approximations of the primal and dual solutions. In (a), the blue line for $L_0(u_0) + R_L^\defbd(u_0, z_1)$ is hidden behind the one for $L_0(u_1) + R_L^\defbd(u_1, z_0)$. In addition, in (b), the convergence rates $\mathcal{O}(|\log|\defbd||^{-1/2})$ and $\mathcal{O}(|\defbd|)$ are plotted for the zeroth- and first-order approximations, respectively.}
    \label{fig:poisson:green:numerical results}
\end{figure}
In \cref{fig:poisson:green:qoi values}, the values of the QoI applied to the exact solution $u$, the defeatured solution $u_0$ and the corrected defeatured solution $u_1$ are plotted for different feature sizes $|\defbd| \in [10^{-3}, \tfrac{1}{4}]$. We observe that $L_0(u_0)$ remains significantly distant from $L(u)$ across all feature sizes, whereas $L_0(u_1)$ provides a substantially better approximation.

Furthermore, \cref{fig:poisson:green:qoi values} also shows the corrected QoIs using the corrector term $R_L^\defbd$ applied to different combinations of $u_0, u_1, z_0$ and $z_1$. We observe that when only the zeroth-order approximations $u_0$ and $z_0$ are used to reconstruct $L(u)$, the corrector term has no significant impact. This is due to the extremely slow convergence of the zeroth-order approximations in the limit $|\defbd| \to 0$. Indeed, in \cref{fig:poisson:green:primal dual errors}, we see that $u_0$ and $z_0$ deviate more than 10\% from their respective exact counterparts in the energy norm, even for the smallest feature sizes. Hence, they do not contain enough information for the true QoI to be reconstructed.
In contrast, using the first-order approximation in either the primal or dual variable is enough to accurately reconstruct the exact QoI value $L(u)$ using the corrector term $R_L^\defbd$. This is reflected by the much faster convergence rate in the limit $|\defbd| \to 0$ for the first-order approximations in \cref{fig:poisson:green:primal dual errors}.

Clearly, the numerical approximation of the first-order approximations becomes increasingly challenging as the feature size approaches zero due to the singularity of the Green's function $G$, which manifests itself in the increase of their relative errors for small feature sizes in \cref{fig:poisson:green:primal dual errors}.

In summary, the goal-oriented estimator is reliable regardless of the energy-norm errors in the primal and dual variables. However, for Dirichlet features the corrector improves on the defeatured QoI only when at least one of its arguments (primal or dual) is accurate in the energy norm. Since $u_0$ and $z_0$ both converge at the slow logarithmic rate, the corrected QoI $L_0(u_0) + R^\defbd_L(u_0 ,z_0)$ stays close to the uncorrected $L_0(u_0)$ unless either argument is replaced by its first-order correction.

\subsection{Linear elasticity}

\tdplotsetmaincoords{120}{-15}
\begin{figure}
    \centering
    \input{tikz/elasticity_beam}
    \caption{Illustration of the cantilever beam with a cylindrical hole used in the linear elasticity experiment. The radius of the hole is 0.05 and its $x$-coordinate $x_{\mathrm{hole}}$ is varied between 0.25 and 1.75.}
    \label{fig:elasticity:sketch}
\end{figure}

In this section, we consider the linear elasticity problem for the cantilever beam with a cylindrical hole illustrated in \cref{fig:elasticity:sketch}. In the defeaturing context, the hole is a negative feature on the boundary that changes the topology of the domain.

The Lamé parameters are given by $\mu = 1$ and $\lambda = 1.25$. The beam is clamped at $x = 0$ and subject to gravity. The other boundaries, including the boundary of the hole, are kept free. The QoI in this experiment is the average $z$-displacement at the opposite end of the beam at $x = L$.

\begin{figure}
    \centering
    \begin{subfigure}[T]{0.45\textwidth}
        \includegraphics{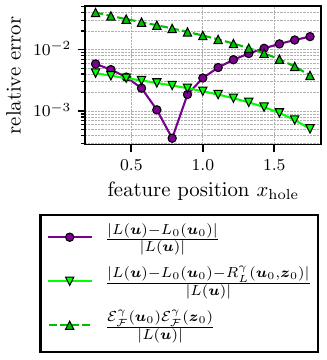}
    \caption{}
    \label{fig:elasticity:numerical results:qoi errors}
    \end{subfigure}
    \hfill
    \begin{subfigure}[T]{0.45\textwidth}
        \includegraphics{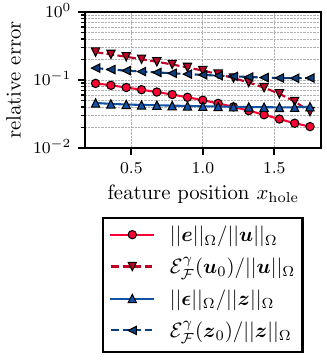}
    \caption{}
    \label{fig:elasticity:numerical results:primal dual errors}
    \end{subfigure}
    \begin{subfigure}[T]{0.45\textwidth}
        \includegraphics{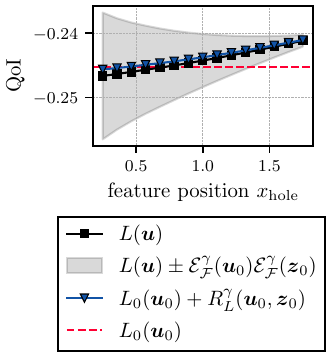}
    \caption{}
    \label{fig:elasticity:numerical results:qoi values}
    \end{subfigure}
    \hfill
    \begin{subfigure}[T]{0.45\textwidth}
        \includegraphics{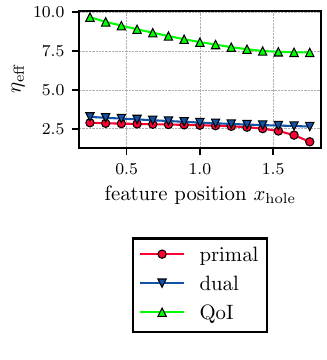}
    \caption{}
    \label{fig:elasticity:numerical results:eff}
    \end{subfigure}
    \caption{Results for the linear elasticity experiment for varying feature positions: (a) relative QoI errors and estimates, (b) relative primal and dual defeaturing errors and estimates, (c) QoI values, (d) effectivity indices. In (b), we write $\boldsymbol{\epsilon} \coloneqq \boldsymbol{z}  - (\bz_0)_{|\domain}$ for the dual defeaturing error.}
    \label{fig:elasticity:numerical results}
\end{figure}

\Cref{fig:elasticity:numerical results} shows the defeaturing results for varying $x$-coordinates of the cylindrical hole $x_{\mathrm{hole}} \in [0.25, 1.75]$. In \cref{fig:elasticity:numerical results:qoi values}, we observe that there is a significant dependence of the exact QoI $L(\bu)$, which is not captured by the defeatured QoI $L_0(\bu_0)$. In contrast, the corrected QoI $L_0(\bu_0) + R_L^\defbd(\bu_0, \bz_0)$ closely follows the exact one.

We also note that the estimate is more conservative when the hole is closer to the clamped boundary, which is explained by stronger stress concentrations around the feature in these cases. Indeed, the defeaturing error in the energy-norm of the primal variable plotted in \cref{fig:elasticity:numerical results:primal dual errors} increases when the feature approaches the clamped boundary. Similarly, the error in the corrected QoI increases in \cref{fig:elasticity:numerical results:qoi errors} when the feature approaches the clamped boundary as $\bu_0$ then provides less and less information through the correction term $R_L^\defbd(\bu_0, \bz_0)$.

In contrast, at both extremes of the parameter interval for the feature position, the defeatured QoI deviates significantly from the actual value. At the same time, there is an equilibrium position $x_{\mathrm{hole}} \approx 0.25$, where the feature does not matter at all. For $x_{\mathrm{hole}}$ close to 1.75, the estimate even becomes unreliable for the uncorrected QoI $L_0(\bu_0)$, highlighting the importance of the corrector term.

In addition, we point out that the effectivity index of the goal-oriented estimate is approximately ten for all feature positions, which is coherent with the effectivity indices of the primal and dual variables; see \cref{fig:elasticity:numerical results:eff}.

Finally, we observe that in contrast to the Dirichlet features in the Poisson experiments, the corrector term $R_L^\defbd(\bu_0, \bz_0)$ adds significant additional information to $L_0(\bu_0)$ that allows us to estimate the impact of the feature position and shape on $L$ without recomputing the exact solution $\bu$.

\subsection{Stokes flow}

\begin{figure}
    \centering
    \begin{subfigure}[T]{0.45\textwidth}
        \input{tikz/stokes_boundary}
    \vspace{-1em}
    \caption{}
    \label{fig:stokes:sketch:boundary}
    \end{subfigure}
    \hfill
    \begin{subfigure}[T]{0.45\textwidth}
        \input{tikz/stokes}
    \vspace{-1em}
    \caption{}
    \label{fig:stokes:sketch:internal}
    \end{subfigure}
    \caption{Illustration of the exact domain $\domain$ for the Stokes experiment with (a)~a half-disk feature $\feat$ cut out of the lower wall and (b)~a disk feature $\feat$ at the center. In both cases, the bottom boundary $\rmbd_{D, \text{no-slip}}$ and the feature boundary are subject to no-slip conditions, while the top boundary $\rmbd_{D, \text{drive}}$ is driven by a constant velocity. The free Neumann boundaries $\rmbd_N$ are marked in red, and the dashed rectangle $S$ marks the integration region for the QoI. In~(a), the dashed straight segment marks the simplified boundary.}
    \label{fig:stokes:sketch}
\end{figure}

In this final set of experiments, we consider Stokes' equations \cref{eq:stokes:exact problem} in the domain $\domain = \defdomain \setminus \overline{\feat}$, with the defeatured domain the rectangle $\defdomain = [-\tfrac{L}{2}, \tfrac{L}{2}] \times [-\tfrac{W}{2}, \tfrac{W}{2}]$, $L = 1$ and $W = \tfrac{1}{2}$. We prescribe vanishing source terms $\bff \equiv \boldsymbol{0}$ and $f_c \equiv 0$. The bottom boundary $\rmbd_{D, \text{no-slip}}$ and the feature boundary are subject to no-slip conditions, the lid $\rmbd_{D, \text{drive}}$ is driven by a constant velocity $U_x = 1$ in the $x$-direction, and the left and right boundaries $\rmbd_N$ are kept free, i.e.\ $\stress(\bu)\bn - p\bn = \boldsymbol{0}$. This yields a linear velocity profile in the defeatured domain, which is perturbed by the feature in the exact domain. The QoI is the average $x$-velocity over $S = [\tfrac{L}{10}, L] \times [-\tfrac{W}{8}, \tfrac{W}{8}]$,
\begin{align*}
    L(\bu) = \int_S u_1 \dd{x}.
\end{align*}
We use this setting to compare the effects of defeaturing caused by a boundary feature and an internal feature, respectively. The two domains are shown in \cref{fig:stokes:sketch}.

\begin{figure}
    \centering
    \captionsetup[subfigure]{aboveskip=2pt, belowskip=0pt} 

    \includegraphics{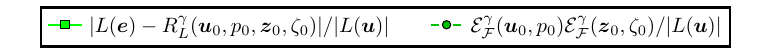}\\[-5pt]   
    \begin{subfigure}[T]{0.45\textwidth}
        \includegraphics{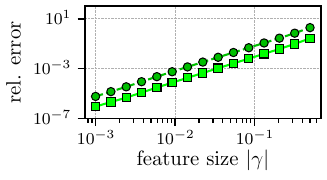}\caption{}
        \label{fig:stokes:numerical results:qoi errors:boundary}
    \end{subfigure}\hfill
    \begin{subfigure}[T]{0.45\textwidth}
        \includegraphics{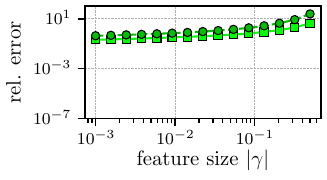}\caption{}
        \label{fig:stokes:numerical results:qoi errors:internal}
    \end{subfigure}\\[2pt]                                            

    \includegraphics{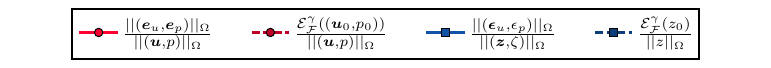}\\[-5pt]
    \begin{subfigure}[T]{0.45\textwidth}
        \includegraphics{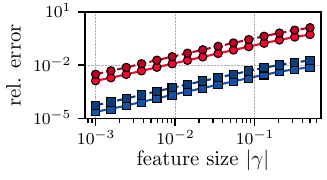}\caption{}
        \label{fig:stokes:numerical results:primal dual errors:boundary}
    \end{subfigure}\hfill
    \begin{subfigure}[T]{0.45\textwidth}
        \includegraphics{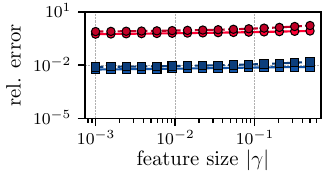}\caption{}
        \label{fig:stokes:numerical results:primal dual errors:internal}
    \end{subfigure}\\[2pt]

    \includegraphics{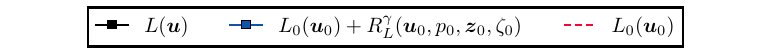}\\[-5pt]
    \begin{subfigure}[T]{0.45\textwidth}
        \includegraphics{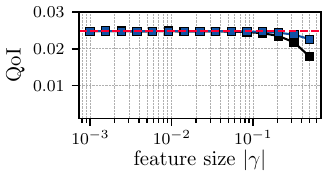}\caption{}
        \label{fig:stokes:numerical results:qoi values:boundary}
    \end{subfigure}\hfill
    \begin{subfigure}[T]{0.45\textwidth}
        \includegraphics{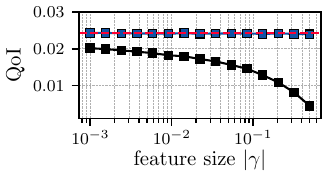}\caption{}
        \label{fig:stokes:numerical results:qoi values}
    \end{subfigure}\\[2pt]

    \includegraphics{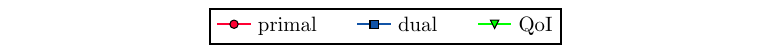}\\[-5pt]
    \begin{subfigure}[T]{0.45\textwidth}
        \includegraphics{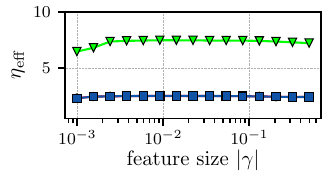}\caption{}
        \label{fig:stokes:numerical results:eff:boundary}
    \end{subfigure}\hfill
    \begin{subfigure}[T]{0.45\textwidth}
        \includegraphics{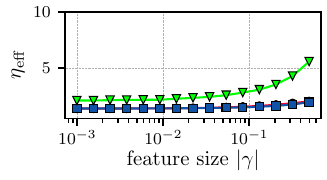}\caption{}
        \label{fig:stokes:numerical results:eff:internal}
    \end{subfigure}
    \caption{Results for the Stokes flow experiments for varying feature size. Left column (a, c, e, g): boundary feature; right column (b, d, f, h): internal feature. From top to bottom: relative QoI errors and estimates (a, b), relative energy-norm defeaturing errors and estimates (c, d), QoI values (e, f), and effectivity indices (g, h). In (c, d), we write $\boldsymbol{\epsilon}_u \coloneqq \boldsymbol{z}  - (\bz_0)_{|\domain}$ and $\epsilon_p \coloneqq \zeta - (\zeta_0)_{|\domain}$.}
    \label{fig:stokes:numerical results}
\end{figure}

\paragraph{A boundary feature}
We first consider a half-disk feature $\feat$ cut out of the lower wall at $x = 0$ (\cref{fig:stokes:sketch:boundary}), with no-slip conditions on the feature boundary as on the lower wall, and vary the feature size in $[10^{-3}, 0.5]$. The results are shown in the left column of \cref{fig:stokes:numerical results}. Both the goal-oriented and the energy-norm estimators are reliable (\cref{fig:stokes:numerical results:qoi errors:boundary,fig:stokes:numerical results:primal dual errors:boundary}), and the relative errors in the primal and dual variables, as well as in the QoI, decrease rapidly as the feature size approaches zero. Moreover, the effectivity index is asymptotically independent of the feature size, as predicted by the theory (\cref{fig:stokes:numerical results:eff:boundary}).

\paragraph{An internal feature}
Next, we consider a disk feature $\feat$ at the center of the domain (\cref{fig:stokes:sketch:internal}), again with no-slip conditions on its boundary, and vary the feature size over the same interval. The results are collected in the right column of \cref{fig:stokes:numerical results}. As for the boundary feature, both estimators are reliable (\cref{fig:stokes:numerical results:qoi errors:internal,fig:stokes:numerical results:primal dual errors:internal}) and the effectivity index is asymptotically independent of the feature size (\cref{fig:stokes:numerical results:eff:internal}). In contrast, the error in the QoI no longer decreases rapidly and remains above $10\%$ even for the smallest feature size. Moreover, as in the Poisson example with an interior feature (cf.\ \cref{fig:poisson:green:qoi values}), the corrector term does not provide additional information for the QoI reconstruction.

We conclude that the defeaturing error in the QoI does not necessarily vanish as the feature size approaches zero, and that whether it does depends on the combination of the feature position, the boundary conditions, and the QoI. Hence, geometric simplifications based solely on feature size are, in general, inadequate for guaranteeing a prescribed error threshold for the QoI.

%% file: tikz/poisson_l2.tex
\begin{tikzpicture}
    \def\length{3.5}
    \def\width{3.5}
    \def\featRadius{0.7}
    \def\qoiSidelength{0.6}


    \coordinate (A) at (0, 0);
    \coordinate (B) at (\length, 0);
    \coordinate (C) at (\length, \width);
    \coordinate (D) at (0, \width);

    \coordinate (m) at (\length / 2, \width);
    \coordinate (i1) at (\length / 2 - \featRadius, \width);
    \coordinate (i2) at (\length / 2 + \featRadius, \width);

    \coordinate (a) at (\length / 2 - \qoiSidelength / 2, \width / 2 - \qoiSidelength / 2);
    \coordinate (b) at (\length / 2 + \qoiSidelength / 2, \width / 2 - \qoiSidelength / 2);
    \coordinate (c) at (\length / 2 + \qoiSidelength / 2, \width / 2 + \qoiSidelength / 2);
    \coordinate (d) at (\length / 2 - \qoiSidelength / 2, \width / 2 + \qoiSidelength / 2);

    \node at (\length / 5, \width / 2) {$\domain$};

    \draw[thick] (A) -- (B) -- (C) -- (i2);
    \draw[thick] (i1) -- (D) -- (A);

    \filldraw[pattern=dots, pattern color=gray, thick] (m) ++ (180:\featRadius) arc (180:360:\featRadius);
    \draw[dashed] (i1) -- (i2);

    \node[fill=white] at (\length / 2, \width - \featRadius / 2) {$\feat$};
    \node[anchor=north] at (\length / 2, \width - \featRadius) {$\defbd$};
    \node[anchor=south] at (\length / 2, \width) {$\simpbd$};
    
    \draw[thick, dashed] (a) -- (b) -- (c) -- (d) -- cycle;

    \node[fill=white] at (\length / 2, \width / 2) {$S$};

\end{tikzpicture}

%% file: tikz/poisson_green.tex
\begin{tikzpicture}
    \def\length{3.5}
    \def\width{3.5}
    \def\featRadius{0.5}


    \coordinate (A) at (0, 0);
    \coordinate (B) at (\length, 0);
    \coordinate (C) at (\length, \width);
    \coordinate (D) at (0, \width);

    \coordinate (F) at (\length / 2, 3 * \width / 4);
    
    \coordinate (a) at (\length / 4, \width / 10);
    \coordinate (b) at (3 * \length / 4, 2 * \width / 5);

    \node at (\length / 5, \width / 2) {$\domain$};

    \draw[thick] (A) -- (B) -- (C) -- (D) -- cycle;

    \filldraw[pattern=dots, pattern color=gray, thick] (F) circle [radius=\featRadius];

    \node[fill=white] at (F) {$\feat$};
    \node[] at (\length / 2 + \featRadius, 3 * \width / 4 + \featRadius) {$\defbd$};
    
    \draw[thick, dashed] (a) rectangle (b);

    \node[fill=white] at (\length / 2,  2.5 *\width / 10) {$S$};

\end{tikzpicture}

%% file: tikz/elasticity_beam.tex
\begin{tikzpicture}[
    tdplot_main_coords,
    scale=3,
    font=\sffamily,
    dimline/.style={-Latex, gray}
]
    \def\L{2}    
    \def\W{0.5}  
    \def\H{0.5}  

    \coordinate (O)   at (0,-\W / 2,-\H / 2);
    \coordinate (A)   at (\L,-\W / 2,-\H / 2);
    \coordinate (B)   at (\L,\W / 2,-\H / 2);
    \coordinate (C)   at (0,\W / 2,-\H / 2);
    \coordinate (O_t) at (0,-\W / 2,\H / 2);
    \coordinate (A_t) at (\L,-\W / 2,\H / 2);
    \coordinate (B_t) at (\L,\W / 2,\H / 2);
    \coordinate (C_t) at (0,\W / 2,\H / 2);

    \def\R{0.1} 
    \def\holeX{\L/3}
    \def\holeZ{0}

    \draw[
        thick,
        fill=gray!20,
        pattern color=gray!60
    ]
        (0, -0.25 - \W / 2, -0.25 - \H / 2) -- (0, \W / 2+0.25, -0.25 - \H / 2) -- (0, \W / 2+0.25, \H / 2+0.25) -- (0, -0.25 - \W / 2, \H / 2+0.25) -- cycle;


    \filldraw[draw=black, thick, fill=blue1!45] (O_t) -- (C_t) -- (B_t) -- (A_t) -- cycle; 
    \filldraw[draw=black, thick, fill=blue1!30] (A) -- (B) -- (B_t) -- (A_t) -- cycle;   
    
    \fill[blue1!15, even odd rule] 
        (C) -- (B) -- (B_t) -- (C_t) -- cycle 
        [canvas is xz plane at y=\W/2] (\holeX, \holeZ) circle (\R); 

    \draw[thick] (C) -- (B) -- (B_t) -- (C_t) -- cycle;
    \filldraw[thick, canvas is xz plane at y=\W/2, fill=blue1!2] (\holeX, \holeZ) circle (\R);

    \draw[thick, dashed] (O) -- (A);
    \draw[thick, dashed] (O) -- (C);
    \draw[thick, dashed] (O) -- (O_t);
    \draw[thick, dashed, canvas is xz plane at y=-\W/2] (\holeX, \holeZ) circle (\R);

    \draw[thick, dashed] (\holeX - \R, \W / 2, \holeZ) -- (\holeX - \R, -\W / 2, \holeZ);
    \draw[thick, dashed] (\holeX + \R, \W / 2, \holeZ) -- (\holeX + \R, -\W / 2, \holeZ);

    \draw[dimline, <->] (0, \W / 2 + 0.1, - \H / 2) -- (\L, \W / 2 + 0.1, - \H / 2) node[midway, below] {$L = 2$};
    
    \draw[dimline, <->] (\L+0.05, -\W / 2, -\H / 2) -- (\L+0.05, -\W / 2, \H / 2) node[midway, right] {$h = 0.5$};

    \draw[dimline, <->] (\L + 0.05, -\W / 2, -\H / 2) -- (\L + 0.05, \W / 2, -\H / 2) node[midway, right] {$w = 0.5$};

    \def\crosssize{0.0175} 
    \draw[gray, dashed] (\holeX, \W/2, \holeZ) -- (\holeX, \W/2, -\H/2);
    \draw[gray] (\holeX - \crosssize, \W/2, \holeZ - \crosssize) -- (\holeX + \crosssize, \W/2, \holeZ + \crosssize);
    \draw[gray] (\holeX + \crosssize, \W/2, \holeZ - \crosssize) -- (\holeX - \crosssize, \W/2, \holeZ + \crosssize);
    \draw[gray] (\holeX - \crosssize, \W/2, -\H / 2 - \crosssize) -- (\holeX + \crosssize, \W/2, -\H / 2 + \crosssize);
    \draw[gray] (\holeX + \crosssize, \W/2, -\H / 2 - \crosssize) -- (\holeX - \crosssize, \W/2, -\H / 2 + \crosssize);

    \draw[-Latex, blue1, <->] (0, \W / 2 + 0.2, - \H / 2) -- (\holeX, \W / 2 + 0.2, - \H / 2) node[midway, below] {$x_{\mathrm{hole}}$};
    
    \draw[-{Stealth[length=3mm]}, thick, red1] (0,0,0) -- (2.3,0,0) node[anchor=north east]{$x$};
    \draw[-{Stealth[length=3mm]}, thick, green1!60!black] (0,0,0) -- (0, \W / 2 + 0.75,0) node[anchor=north west]{$y$};
    \draw[-{Stealth[length=3mm]}, thick, blue1] (0,0,0) -- (0,0,\H / 2 + 0.75) node[anchor=south]{$z$};

\end{tikzpicture}

%% file: tikz/stokes_boundary.tex
\begin{tikzpicture}
    \def\length{5}
    \def\width{2.5}
    \def\featRadius{0.4}
    \def\QoIOffset{0.5}

    \coordinate (BL) at (0, 0);
    \coordinate (BR) at (\length, 0);
    \coordinate (TR) at (\length, \width);
    \coordinate (TL) at (0, \width);
    \coordinate (IL) at (\length / 2 - \featRadius, 0);  
    \coordinate (IR) at (\length / 2 + \featRadius, 0);  

    \fill[pattern=dots, pattern color=gray]
        (IL) arc[start angle=180, end angle=0, radius=\featRadius] -- cycle;

    \draw[thick] (BL) -- (IL);                 
    \draw[thick] (IR) -- (BR);                 
    \draw[thick] (TR) -- (TL);                 
    \draw[thick, red1] (BR) -- (TR);           
    \draw[thick, red1] (TL) -- (BL);           

    \draw[thick, dashed] (IL) -- (IR);

    \draw[thick] (IL) arc[start angle=180, end angle=0, radius=\featRadius];

    \node at (\length / 5, \width / 2) {$\domain$};

    \node[fill=white, inner sep=1pt] at (\length / 2, \featRadius / 2) {$\feat$};

    \node[anchor=south] at (\length / 2, \width) {$\rmbd_{D, \text{drive}}$};
    \node[anchor=north] at (\length / 2, 0) {$\rmbd_{D, \text{no-slip}}$};
    \node[anchor=east, red1] at (0, \width / 2) {$\rmbd_{N}$};
    \node[anchor=west, red1] at (\length, \width / 2) {$\rmbd_{N}$};

    \coordinate (q1) at (\length / 2 + \QoIOffset, \width / 2 - \width / 8);
    \coordinate (q2) at (\length, \width / 2 + \width / 8);
    \draw[thick, dashed] (q1) rectangle (q2);
    \node at ($(q1)!0.5!(q2)$) {$S$};

\end{tikzpicture}

%% file: tikz/stokes.tex
\begin{tikzpicture}
    \def\length{5}
    \def\width{2.5}
    \def\featRadius{0.4}
    \def\QoIOffset{0.5}
    
    \draw[thick] (0, 0) rectangle (\length, \width);

    \node at (\length / 5, \width / 2) {$\domain$};

    \filldraw[thick, pattern=dots, pattern color=gray]  (\length / 2, \width / 2) circle [radius=\featRadius];

    \node[fill=white] at (\length / 2, \width / 2) {$\feat$};

    \node[anchor=south] at (\length / 2, \width) {$\rmbd_{D, \text{drive}}$};
    \node[anchor=north] at (\length / 2, 0) {$\rmbd_{D, \text{no-slip}}$};

    \node[anchor=east, red1] at (0, \width / 2) {$\rmbd_{N}$};
    \node[anchor=west, red1] at (\length, \width / 2) {$\rmbd_{N}$};

    \draw[thick, red1] (0, 0) -- (0,\width);
    \draw[thick, red1] (\length, 0) -- (\length,\width);

    \coordinate (q1) at (\length / 2 + \QoIOffset, \width / 2 - \width / 8);
    \coordinate (q2) at (\length, \width / 2 + \width / 8);
    \draw[thick, dashed] (q1) rectangle (q2);
    \node at ($(q1)!0.5!(q2)$) {$S$};

\end{tikzpicture}

%% file: contents/07_conclusion.tex
\section{Conclusion}
\label{sec:conclusion}
This work establishes a mathematically certified framework for goal-oriented a posteriori error control in geometric defeaturing. By combining the dual-weighted residual (DWR) method\cite{becker_optimal_2001,prudhomme_goal-oriented_1999} with rigorous energy-norm estimates, we derive certified error bounds for linear quantities of interest (QoIs). Our contributions are threefold: we established new reliable energy-norm estimators for Dirichlet features in linear elasticity and Stokes flow; we formulated unified estimators for geometries with mixed feature types; and we derived a certified goal-oriented framework applicable to the Poisson, linear elasticity, and Stokes problems that inherits this generality.

The current framework is limited to negative features. Based on existing estimators for positive Neumann features\cite{chanon_adaptive_2022}, the corresponding goal-oriented estimates could be derived similarly. Positive Dirichlet features, however, first require the derivation of energy-norm estimates. Additionally, the current framework only provides an upper bound on the error. While the energy-norm estimates for the Neumann features are known to be efficient\cite{buffa_analysis-aware_2022,chanon_adaptive_2022}, a suitable lower bound remains to be established for Dirichlet features. Hence, future work includes proving the efficiency of the Dirichlet estimators and developing a suitable framework for positive Dirichlet defeaturing.

The goal-oriented error estimation framework presented in this paper not only establishes a reliable error estimate but also offers a method for reconstructing the exact quantity of interest. This reconstruction is based solely on the defeatured primal and dual solutions, along with the feature boundary information. This approach could open new avenues for shape optimization.

%% file: contents/A_results_on_trace_spaces.tex
\section{Some results in Sobolev trace spaces}
\label{sec:results on trace spaces}
In this section, we collect definitions and results in Sobolev trace spaces that are useful in the context of defeaturing error estimation. For details and proofs, we refer to Chanon\cite{chanon_adaptive_2022}, Weder and Buffa\cite{weder_analysis-aware_2025} and Weder\cite{weder_extension_2025}. Moreover, these definitions and results will be used in the proofs in \cref{sec:linear elasticity:reliability,sec:stokes:reliability}.

The following subspace of the trace space $\htrace{\genbd}$ is essential for our analysis:
\begin{align}
\label{eq:definition htracedbz}
    \htracedbz{\genbd} \coloneqq \left\{\mu \in \Ltwo{\genbd} : \mu^\star \in \htrace{\partial \domain}\right\},
\end{align}
where we write $\mu^\star$ for the extension of $\mu$ by 0 on $\partial \domain$. The corresponding norm and semi-norm are, respectively, defined by
\begin{align*}
    \htracedbznorm{\mu}{\genbd}^2 &\coloneqq \norm{\mu}{0}{\genbd}^2 + \htracedbzseminorm{\mu}{\genbd}^2,
    \\
    \htracedbzseminorm{\mu}{\genbd}^2 &\coloneqq \htraceseminorm{\mu}{\genbd}^2 + 2\int_\genbd \int_{\partial \domain \setminus \genbd} \frac{|\mu(y)|^2}{|x - y|^{n}} \dd{s}(x) \dd{s}(y).
\end{align*}
In particular, we have
\begin{align*}
    \htracedbznorm{\mu}{\genbd} = \htracenorm{\mu^\star}{\partial \domain} & &\text{ and } & & \htracedbzseminorm{\mu}{\genbd} = \htraceseminorm{\mu^\star}{\partial\domain}.
\end{align*}
These definitions generalize readily to the vector-valued versions of these spaces.

The following lemma from Weder and Buffa\cite{weder_analysis-aware_2025} provides an estimate for the weakly singular kernel $y \mapsto |x - y|^{-n}$. It is used to bound the double integrals in the proof of \Cref{thm:linear elasticity:dirichlet estimator reliability}:

\begin{lemma}[Lemma A.2 in Weder and Buffa\cite{weder_analysis-aware_2025}]
\label{lemma:kernel integral estimate}
    Let $\domain \subset \R^n$ be a Lipschitz domain. Then, there exists a constant $C > 0$ depending only on the Lipschitz character of $\partial \domain$ such that for any $x \in \domain$,
    \begin{align*}
        \int_{\partial \domain} \frac{1}{|x - y|^n} \dd{s}(y) \leq \frac{C}{\dist(x, \partial \domain)}.
    \end{align*}
\end{lemma}

For subsets of the boundary $\genbd \subset \partial \domain$ with $|\genbd| > 0$, we can naturally extend the definition of the Neumann trace operator in \cref{eq:defnition of neumann operator}  to $\partialneumop \genbd: \hdiv\domain \to \htracedbzdual{\genbd}$ by setting
\begin{align}
\label{eq:definition of partial neumann operator}
    \langle \partialneumop\genbd(v), \mu\rangle \coloneqq \langle \neumop(v), \mu^\star \rangle, & & v \in \hdiv\domain,\, \mu \in \htracedbz{\genbd}.
\end{align}

Finally, we point out that one can equip the trace space $\htrace{\partial \domain}$ with the equivalent natural trace norm given by
\begin{align}
\label{eq:preliminaries:definition natural trace norm}
\htraceinfnorm{\mu}{\partial \domain} \coloneqq \inf_{u \in \honebd{\mu}{\partial \domain}{\domain}}\honenorm{u}{\domain} & & \forall \mu \in \htrace{\partial \domain}.
\end{align}
We refer to Hsiao and Wendland\cite{hsiao_boundary_2021} for a detailed discussion of equivalent trace space norms.

The following Poincaré and interpolation-type inequalities in trace spaces are crucial for the derivation of the Dirichlet defeaturing estimators. Although stated for scalar-valued function spaces, these inequalities readily extend to their vector-valued counterparts.

\begin{lemma}[Lemma 2.3.6. in Chanon\cite{chanon_adaptive_2022}]
\label{lemma:poincare I}
Assume that $\genbd \subset \partial \domain$  for some Lipschitz domain $\domain \subset \R^n$. Assume that $\genbd$ is isotropic according to \cref{def:isotropic family} and connected, and $\partial \genbd \neq \emptyset$. Then, for all $\mu \in \htracedbz{\genbd}$,
\begin{align*}
    \Ltwonorm{\mu}{\genbd} \lesssim |\genbd|^{\frac{1}{2(n-1)}} \htraceseminorm{\mu^\star}{\partial \domain} \leq |\genbd|^{\frac{1}{2(n-1)}} \htracedbznorm{\mu}{\genbd}.
\end{align*}
\end{lemma}

\begin{lemma}[Lemma 2.3.8. in Chanon\cite{chanon_adaptive_2022}]
\label{lemma:poincare II}
    Assume that $\genbd$ is isotropic according to \cref{def:isotropic family}. Then, for all $\mu \in \htrace{\genbd}$,
    \begin{align*}
        \Ltwonorm{\mu - \avg{\mu}{\genbd}}{\genbd} \lesssim |\genbd|^{\frac{1}{2(n-1)}} \htraceseminorm{\mu}{\genbd}.
    \end{align*}
\end{lemma}

\begin{lemma}[Lemma A.5 in Weder and Buffa\cite{weder_analysis-aware_2025}]
\label{lemma:interpolation and poincare inequality}
    Assume that $\genbd$ is isotropic according to \cref{def:isotropic family}. Then, for all $\mu \in \hone{\genbd}$,
    \begin{align*}
        \htraceseminorm{\mu - \avg{\mu}{\genbd}}{\genbd} \lesssim \sqrt{\Ltwonorm{\mu - \avg{\mu}{\genbd}}{\genbd} \Ltwonorm{\nabla_t \mu}{\genbd}},
    \end{align*}
    where $\nabla_t \mu$ denotes the gradient in tangential direction. 
    Moreover, if $\mu \in \htracedbz{\genbd} \cap \hone{\genbd}$, we have
    \begin{align*}
        \htracedbzseminorm{\mu}{\genbd} \lesssim \sqrt{\Ltwonorm{\mu}{\genbd} \Ltwonorm{\nabla_t\mu}{\genbd}}.
    \end{align*}
\end{lemma}

Next, we state some results on the norm equivalence between the natural trace space norm defined in \cref{eq:preliminaries:definition natural trace norm} and the integral trace norm defined in \cref{eq:definition of fractional sobolev norm}. A detailed discussion of different trace space norms can be found in \cite{hsiao_boundary_2021}.

\begin{proposition}[Proposition 2.6 in Weder and Buffa\cite{weder_analysis-aware_2025}]
\label{prop:norm equivalence for htracedbz}
    Let $\domain \subset \R^n, n \geq 2$ be a domain with Lipschitz boundary $\partial \domain$ and $\genbd \subset \partial \domain$ with $|\genbd| > 0$. Then, for every $\mu \in \htracedbz{\genbd}$,
    \begin{align*}
        \htraceinfnorm{\mu^\star}{\partial \domain} \lesssim |\genbd|^{\frac{-1}{2(n-1)}} \Ltwonorm{\mu}{\genbd} + \htracedbzseminorm{\mu}{\genbd}.
    \end{align*}
\end{proposition}

\begin{lemma}[Lemma 4.3 in Weder and Buffa\cite{weder_analysis-aware_2025}]
\label{lemma:natural norm estimate for constants}
Let $\domain = \defdomain \setminus \overline{\feat} \subset \R^n, n \geq 2$ be a Lipschitz domain with internal Lipschitz feature $\feat$ and feature boundary $\defbd = \partial \feat$. We write $\rmbd = \partial \defdomain$ and we denote by $c_\feat$ and $\rcirc(\feat)$ the circumcenter and circumradius of $\feat$, respectively.

Assume that $\rcirc(\feat) < \dist(c_\feat, \rmbd)$. Then, it holds for $1 \in \htracedbz{\defbd}$ that,
\begin{align*}
    \htraceinfnorm{1^\star}{\partial \domain} \lesssim \Bar{c}_\defbd,
\end{align*}
where
\begin{align*}
    \Bar{c}_\defbd \coloneqq  
        \begin{cases}
            \sqrt{\frac{2 \pi}{|\log\left(s_\defbd\right)|}}, & n = 2,\\
            \sqrt{\frac{4 \pi \rcirc(\feat)}{1 - s_\defbd}}, & n=3
        \end{cases}, && \text{ and } && s_\defbd \coloneqq \frac{\rcirc(\feat)}{\dist(c_\feat, \rmbd)}.
\end{align*}
\end{lemma}

A generalization of the following result on the operator norm of the partial Neumann trace operator will be used repeatedly in \cref{sec:linear elasticity:reliability,sec:stokes:reliability}:

\begin{lemma}[Lemma 2.7 in Weder and Buffa\cite{weder_analysis-aware_2025}]
\label{lemma:size independence of neumann trace operator norm}
    Let $\domain \subset \R^n, n \geq 2$ be a Lipschitz domain and $\genbd \subset \partial \domain$ with $|\genbd| > 0$ and $\partial\genbd \neq \emptyset$. Then, it holds for any $v \in \hdiv\domain$ and $\mu \in \htracedbz{\genbd}$ that
    \begin{align*}
        \langle \partialneumop\genbd (v), \mu \rangle \lesssim \hdivnorm{v}{\domain} \htracedbznorm{\mu}{\genbd}.
    \end{align*}
\end{lemma}

Finally, we recall the following extension result on Lipschitz boundaries adapted to the context of defeaturing error estimation, which is essential for the treatment of Dirichlet-Neumann features. For a detailed proof, we refer to Weder\cite{weder_extension_2025}.

\begin{theorem}[Theorem A.7 in Weder and Buffa\cite{weder_analysis-aware_2025}]
\label{thm:boundary extension operator}
    Let $\genbd \subset \R^n$ be a compact $(n-1)$-dimensional Lipschitz manifold and $\defbd \subset \genbd$ an $(n-1)$-dimensional submanifold with boundary. Furthermore, we assume that there exists a subset $\genbd_0 \subset \genbd$, such that
    \begin{align*}
        \defbd \subset \genbd_0 \subset \genbd, & \text{ and } \partial \defbd \cap \partial \genbd = \emptyset.
    \end{align*}
    Finally, we assume that the boundary $\partial \defbd$ of $\defbd$ is an $(n-2)$-dimensional Lipschitz manifold.

    Then, there exists a continuous extension operator $\extop{\defbd}{\genbd}: \htrace{\defbd} \to \htracedbz{\genbd}$, such that for $\mu \in \htrace{\defbd}$,
    \begin{align*}
        \Ltwonorm{\extop{\defbd}{\genbd}(\mu)}{\genbd}^2 \lesssim \Ltwonorm{\mu}{\defbd}^2,
    \end{align*}
    and
    \begin{align*}
        \htracedbzseminorm{\extop{\defbd}{\genbd}(\mu)}{\genbd}^2 \lesssim |\defbd|^{\frac{-1}{n-1}} \Ltwonorm{\mu}{\defbd}^2 + \htraceseminorm{\mu}{\defbd}^2.
    \end{align*}
\end{theorem}

%% file: contents/A_linear_elasticity.tex
\section{The estimator for Dirichlet features in linear elasticity}
\label{sec:linear elasticity:reliability}
This appendix is dedicated to proving the reliability of the a posteriori error estimators for Dirichlet features in the linear elasticity problem \cref{eq:elasticity:dirdir feature estimator,eq:elasticity:dirneum feature estimator,eq:elasticity:dirint feature estimator}, as stated in \cref{thm:linear elasticity:dirichlet estimator reliability}. The proof follows the same strategy as the one for the Poisson equation presented in Weder and Buffa\cite{weder_analysis-aware_2025}, using key estimates for Sobolev trace spaces from \cref{sec:results on trace spaces}.

To that end, it suffices to consider the exact problem \cref{eq:linear elasticity:exact problem} and the defeatured problem \cref{eq:linear elasticity:defeatured problem} with one single feature of Dirichlet-Dirichlet, Dirichlet-Neumann or internal Dirichlet type. In this simplified case, the defeaturing error satisfies the following PDE:
\begin{equation}
\label{eq:linear elasticity:reliability:error pde}
\begin{cases}
    - \nabla \cdot \stress(\be) = \boldsymbol{0}, & \text{ in } \domain,\\
    \stress(\be) \bn = \boldsymbol{0}, & \text{ on } \rmbd_N,\\
    \be = \boldsymbol{0}, & \text{ on } \rmbd_D \setminus \overline{\defbd},\\
    \be = \bderrvec, & \text{ on } \defbd.
\end{cases}
\end{equation}
In contrast to the main text, it is enough to assume here that $\bderrvec \in \honevec{\defbd}$ for the tangential gradient $\nabla_t \bderrvec$ to be in $\Ltwovec{\defbd}$, while no further assumptions are necessary for $\stress(\be) \bn$.
From an integration by parts and \cref{eq:linear elasticity:reliability:error pde}, we find that
\begin{align}
\label{eq:linear elasticity:error ibp}
    \energynorm{\be}{\domain}^2 = \int_\domain \stress(\be) : \strain(\be) \dd{x} = \int_{\partial \domain} \stress(\be) \bn \cdot \be \dd{s},
\end{align}
To arrive at an error estimate, we must carefully study the Neumann trace $\stress(\be) \bn$ on the right-hand side of \cref{eq:linear elasticity:error ibp}.
Accordingly, we define the space
\begin{align*}
\hdivtensor{\domain} \coloneqq \{\bs \in \Ltwotensor{\domain}: \nabla \cdot \bs \in \Ltwovec{\domain}\},
\end{align*}
equipped with the norm
\begin{align*}   
\hdivtensornorm{\bs}{\domain}^2 \coloneqq \Ltwonorm{\bs}{\Omega}^2 + \Ltwonorm{\nabla \cdot \bs}{\Omega}^2.
\end{align*}
Then, the Neumann operator for the linear elasticity problem, $\neumopvec: \hdivtensor{\domain} \to \htracedualvec{\partial\domain}$, is defined by integration by parts,
\begin{align}
\label{eq:linear elasticity:definition of the neumann trace}
    \langle \neumopvec(\bs), \bmu \rangle \coloneqq \int_\domain (\nabla \cdot \bs) \cdot \boldsymbol{\phi}_{\bmu} \dd{x} + \int_\domain \bs : \strain(\boldsymbol{\phi}_{\bmu}) \dd{x},
\end{align}
for some $\boldsymbol{\phi}_{\bmu} \in \honevecbd{\domain}{\bmu}{\partial\domain}$. This construction is similar to the one for the standard Neumann operator in Girault and Raviart\cite{girault_finite_1986}. This operator is clearly continuous since the Cauchy-Schwarz inequality yields
\begin{align*}
    \langle \neumopvec(\bs), \bmu \rangle &\leq \Ltwonorm{\nabla \cdot \bs}{\domain} \Ltwonorm{\boldsymbol{\phi}_{\bmu}}{\domain}
    \\
    &+ \Ltwonorm{\bs}{\domain} \honeseminorm{\boldsymbol{\phi}_{\bmu}}{\domain} \leq \hdivtensornorm{\bs}{\domain} \honenorm{\boldsymbol{\phi}_{\bmu}}{\domain}.
\end{align*}
Taking the infimum over all possible liftings $\boldsymbol{\phi}_{\bmu}$, we obtain
\begin{align*}
    \langle \neumopvec(\bs), \bmu \rangle \leq \hdivtensornorm{\bs}{\domain} \htraceinfnorm{\bmu}{\partial \domain},
\end{align*}
which proves the continuity with respect to the equivalent natural trace space norm \cref{eq:preliminaries:definition natural trace norm}. Given a subset $\genbd \subset \partial\domain$ with $|\genbd| > 0$, we may define the partial Neumann operator onto $\genbd$ as $\partialneumopvec{\genbd}: \hdivtensor{\domain} \to \htracedbzdualvec{\genbd}$ with
\begin{align}
\label{eq:linear elasticity:definition of the partial neumann trace}
    \langle \partialneumopvec{\genbd}(\bs), \bmu \rangle \coloneqq \langle \neumopvec(\bs), \bmu^\star \rangle,
\end{align}
where $\bmu^\star$ denotes the extension by zero of $\bmu$ to $\htracevec{\partial \domain}$.
\begin{remark}
\label{rem:elasticity:neumop operator norm}
We point out that if $\genbd$ itself has a boundary, then a result similar to \cref{lemma:size independence of neumann trace operator norm} holds for $\partialneumopvec{\genbd}$ asserting that its operator norm does not depend on the size of $\genbd$.
\end{remark}

Starting from the error representation in \cref{eq:linear elasticity:error ibp} and using the Neumann operator introduced in the previous subsection, we have
\begin{align*}
    \energynorm{\be}{\domain}^2 = \langle \neumopvec(\stress(\be)), \be \rangle.
\end{align*}
For Dirichlet-Dirichlet and internal features, we have $\bderrvec \in \htracedbzvec{\defbd} \cap \honevec{\defbd}$ and in particular
\begin{align}
\label{eq:linear elasticity:dd and internal boundary rep}
    \energynorm{\be}{\domain}^2 = \langle \partialneumopvec{\defbd}(\stress(\be)), \bderrvec \rangle.
\end{align}
For Dirichlet-Neumann features, the boundary error $\bderrvec$ is not in $\htracedbzvec{\defbd} \cap \honevec{\defbd}$ in general, but merely in $\htracevec{\defbd} \cap \honevec{\defbd}$. Therefore, the previous boundary representation of the error is not valid. However, since the Neumann trace of $\stress(\be)$ vanishes on $\rmbd_N$, we can extend $\bderrvec$ to $\htracedbzvec{\genbd}$, where 
\begin{align*}
\defbd \subset \genbd \subset \interior{\overline{\defbd} \cup \overline{\rmbd_N}},
\end{align*}
is an open neighborhood of $\defbd$ in the Neumann boundary. Using the continuous extension operator $\extop{\defbd}{\genbd}: \htracevec{\defbd} \to \htracedbzvec{\genbd}$ from \cref{thm:boundary extension operator}, we may thus write
\begin{align}
\label{eq:linear elasticity:dn boundary rep}
\energynorm{\be}{\domain}^2 = \langle \partialneumopvec{\genbd}(\stress(\be)), \extop{\defbd}{\genbd}(\bderrvec)\rangle.
\end{align}
We circumvented the latter technicality in the main text by assuming that $\neumopvec(\stress(\be)) \in \Ltwovec{\partial \domain}$ due to additional regularity of the solution.
Let us now prove the reliability result for the linear elasticity problem:
\begin{theorem}
\label{thm:linear elasticity:dirichlet estimator reliability}
Consider the defeaturing problem for the linear elasticity equation with one single isotropic Dirichlet feature $\feat$. Then, the following holds:
\begin{enumerate}
    \item If the feature $\feat$ is of Dirichlet-Dirichlet type, then the estimator \cref{eq:elasticity:dirdir feature estimator} is reliable. That is, if $\overline{\defbd} \cap \overline{\rmbd_D} \neq \emptyset$ and $\overline{\defbd} \cap \overline{\rmbd_N} = \emptyset
    $, then
    \begin{align*}
        \energynorm{\be}{\domain} \lesssim \defestdd{\bu_0}{\defbd}.
    \end{align*}
    \item If the feature $\feat$ is of Dirichlet-Neumann type, then the estimator \cref{eq:elasticity:dirneum feature estimator} is reliable. That is, if $\overline{\defbd} \cap \overline{\rmbd_N} \neq \emptyset$, then
    \begin{align*}
        \energynorm{\be}{\domain} \lesssim \defestdn{\bu_0}{\defbd}.
    \end{align*}
    \item If the feature $\feat$ is internal, then the estimator \cref{eq:elasticity:dirint feature estimator} is reliable. That is, if $\overline{\defbd} \cap \overline{(\partial \domain \setminus \defbd)} = \emptyset$ and $\rcirc(\feat) < \dist(c_\feat, \partial \defdomain)$, then
    \begin{align*}
        \energynorm{\be}{\domain} \lesssim \defestint{\bu_0}{\defbd}.
    \end{align*}
\end{enumerate}
\end{theorem}

\begin{remark}
The constants $\tilde{c}_{\defbd}$ and $\Bar{c}_{\defbd}$ play the same role as in the Poisson problem; see \cref{rem:poisson:internal constants} for their practical evaluation.
\end{remark}

\begin{proof}
    To prove assertion 1, we use the boundary representation formula \cref{eq:linear elasticity:dd and internal boundary rep}, \cref{rem:elasticity:neumop operator norm} and \cref{lemma:poincare I} to find
   \begin{align*}
       \energynorm{\be}{\domain}^2 = \langle \partialneumopvec{\defbd}(\stress(\be)), \bderrvec \rangle
       \lesssim 2 \hdivtensornorm{\stress(\be)}{\domain} \htracedbzseminorm{\bderrvec}{\defbd} 
       \lesssim 2 \energynorm{\be}{\domain} \htracedbzseminorm{\bderrvec}{\defbd}.
   \end{align*}
   Applying \cref{lemma:interpolation and poincare inequality} to the last factor and dividing by $\energynorm{\be}{\domain}$ on both sides, we find
   \begin{align*}
       \energynorm{\be}{\domain} \lesssim 2 \sqrt{\Ltwonorm{\bderrvec}{\defbd} \Ltwonorm{\nabla_t \bderrvec}{\defbd}},
   \end{align*}
   which proves the assertion.

   To prove assertion 2, we start from the error representation formula \cref{eq:linear elasticity:dn boundary rep}. Using the same argument as before, we obtain
   \begin{nalign}
   \label{eq:elasticity:dirneum proof first estimate}
       \energynorm{\be}{\domain}^2 &= \langle \partialneumopvec{\genbd}(\stress(\be)),
       \extop{\defbd}{\genbd}(\bd_\defbd - \avg{\bd_\defbd}{\defbd}) + \avg{\bd_\defbd}{\defbd} \extop{\defbd}{\genbd}(1)
       \rangle
       \\
       &\lesssim
       2 \energynorm{\be}{\domain} \left( \htracedbzseminorm{\extop{\defbd}{\genbd}(\bd_\defbd - \avg{\bd_\defbd}{\defbd})}{\genbd} + \vecnorm{\avg{\bd_\defbd}{\defbd}} \htracedbzseminorm{\extop{\defbd}{\genbd}(1)}{\genbd} \right).
   \end{nalign}
   Using the continuity properties of $\extop{\defbd}{\genbd}$ stated in Ref.~\citenum{weder_analysis-aware_2025}[Theorem A.7] and \cref{lemma:poincare II,lemma:interpolation and poincare inequality}, we find
   \begin{nalign}
   \label{eq:elasticity:dirneum proof navg estimate}
       \htracedbzseminorm{\extop{\defbd}{\genbd}(\bd_\defbd - \avg{\bd_\defbd}{\defbd})}{\genbd}^2
       \lesssim |\defbd|^{\frac{-1}{n-1}} \Ltwonorm{\bd_\defbd - \avg{\bd_\defbd}{\defbd}}{\genbd}^2 + \htraceseminorm{\bd_\defbd - \avg{\bd_\defbd}{\defbd}}{\defbd}^2
       \\
       \lesssim \htraceseminorm{\bd_\defbd - \avg{\bd_\defbd}{\defbd}}{\defbd}^2
       \lesssim  \Ltwonorm{\bd_\defbd - \avg{\bd_\defbd}{\defbd}}{\defbd}\Ltwonorm{\nabla_t \bd_\defbd}{\defbd}.
   \end{nalign}
   Similarly, we find
   \begin{align}
   \label{eq:elasticity:dirneum proof avg estimate}
       \htracedbzseminorm{\extop{\defbd}{\genbd}(1)}{\genbd}^2 \lesssim |\defbd|^{\frac{-1}{n-1}}\Ltwonorm{1}{\defbd}^2 \lesssim |\defbd|^{\frac{n - 2}{n-1}}.
   \end{align}
   Inserting the estimates \cref{eq:elasticity:dirneum proof navg estimate} and \cref{eq:elasticity:dirneum proof avg estimate} into \cref{eq:elasticity:dirneum proof first estimate} and simplifying on both sides yields assertion 2.

  To prove assertion 3, we start again from the error representation formula \cref{eq:linear elasticity:dd and internal boundary rep}. However, since the boundary of an internal feature itself does not have a boundary, the Poincaré-type inequality from \cref{lemma:poincare I} does not hold. Hence, the generalization of \cref{lemma:size independence of neumann trace operator norm} does not apply. Nevertheless, we find
   \begin{align*}
        \energynorm{\be}{\domain}^2 &= \langle \partialneumopvec{\defbd}(\stress(\be)), \bderrvec \rangle
        \\
        &\lesssim \hdivtensornorm{\stress(\be)}{\domain}
        \left(
        \htraceinfnorm{(\bd_\defbd - \avg{\bd_\defbd}{\defbd})^\star}{\partial \domain} + \vecnorm{\avg{\bd_\defbd}{\defbd}}\htraceinfnorm{1^\star}{\partial \domain}
        \right ).
   \end{align*}
   Using the fact that $ \hdivtensornorm{\stress(\be)}{\domain} \lesssim \energynorm{\be}{\domain}$, we obtain
   \begin{align}
    \label{eq:elasticity:inf norm expansion}
        \energynorm{\be}{\domain} 
        \lesssim
        \htraceinfnorm{(\bd_\defbd - \avg{\bd_\defbd}{\defbd})^\star}{\partial \domain} + \vecnorm{\avg{\bd_\defbd}{\defbd}}\htraceinfnorm{1^\star}{\partial \domain}.
   \end{align}
    Applying \cref{prop:norm equivalence for htracedbz} and \cref{lemma:poincare II} to the first term, we find
    \begin{align*}
        \htraceinfnorm{(\bd_\defbd - \avg{\bd_\defbd}{\defbd})^\star}{\partial \domain} \lesssim \htraceseminorm{\bd_\defbd - \avg{\bd_\defbd}{\defbd}}{\defbd} + \htracedbzseminorm{\bd_\defbd - \avg{\bd_\defbd}{\defbd}}{\defbd}.
    \end{align*}
    Note that for an internal feature, we have $\dist(\partial \domain \setminus \overline{\defbd}, \defbd) > 0$. Therefore, using \Cref{lemma:kernel integral estimate}, we can estimate the second term as
    \begin{nalign}
    \label{eq:elasticity:interior feature norm equivalence}
        \htracedbzseminorm{\bd_\defbd - \avg{\bd_\defbd}{\defbd}}{\defbd}^2
        = \htraceseminorm{\bd_\defbd - \avg{\bd_\defbd}{\defbd}}{\defbd}^2
        + 2\int_{\partial \domain \setminus \defbd} \int_{\defbd} \frac{|\bd_\defbd(x) - \avg{\bd_\defbd}{\defbd}|^2}{|x - y|^n}\dd{s(x)}\dd{s(y)}
        \\
        \lesssim \htraceseminorm{\bd_\defbd - \avg{\bd_\defbd}{\defbd}}{\defbd}^2
        + \frac{2}{\dist(\partial \domain \setminus \defbd, \defbd)} \Ltwonorm{\bd_\defbd - \avg{\bd_\defbd}{\defbd}}{\defbd}^2
        \\
        \lesssim \left( 1 + \frac{2|\defbd|^{\frac{1}{n-1}}}{\dist(\partial\defdomain, \defbd)} \right) \htraceseminorm{\bd_\defbd - \avg{\bd_\defbd}{\defbd}}{\defbd}^2.
    \end{nalign}
    The second term on the right-hand side of \cref{eq:elasticity:inf norm expansion} can be estimated by \cref{lemma:natural norm estimate for constants}. Inserting this estimate into \cref{eq:elasticity:inf norm expansion} together with \cref{eq:elasticity:interior feature norm equivalence} and applying once more \cref{lemma:interpolation and poincare inequality} as for the two other expressions yields assertion 3.
\end{proof}

%% file: contents/A_stokes.tex
\section{The estimator for Dirichlet features in Stokes' equations}
\label{sec:stokes:reliability}
This appendix presents the proof of reliability for the a posteriori error estimators for Dirichlet features \cref{eq:stokes:dirdir feature estimator,eq:stokes:dirneum feature estimator,eq:stokes:dirint feature estimator} in the Stokes problem, as stated in \cref{thm:stokes:dirichlet estimator reliability}. The proof strategy closely follows that of the linear elasticity case in \cref{sec:linear elasticity:reliability}, with one additional consideration: the defeaturing error in the pressure variable. We first show that the pressure error can be controlled by the velocity error via the inf-sup stability condition of the Stokes problem. With the pressure error bounded, the remainder of the proof focuses on the velocity error and proceeds in direct analogy to the elasticity case, treating each of the three feature types in turn.

We consider the exact problem \cref{eq:stokes:exact problem} and the defeatured problem \cref{eq:stokes:defeatured problem} with one single Dirichlet feature again. Then the defeaturing error functions $\be_u \coloneqq \bu - (\bu_0)_{|\domain}$ and $e_p \coloneqq p - (p_0)_{|\domain}$ in the velocity and pressure variables, respectively, satisfy the PDE problem
\begin{align}
\label{eq:stokes:error pde}
\begin{cases}
    -\nabla \cdot \stress(\be_u) + \nabla e_p = \boldsymbol{0} & \text{ in } \domain,\\
    \nabla \cdot \be_u = 0 & \text{ in } \domain,\\
    \stress(\be_u) \bn - e_p \bn = \boldsymbol{0} & \text{ on } \rmbd_N,\\
    \be_u = \boldsymbol{0} & \text{ on } \rmbd_D \setminus \overline{\defbd},\\
    \be_u = \bderrvec & \text{ on } \defbd,
\end{cases}
\end{align}
with a corresponding weak formulation similar to \cref{eq:stokes:weak form:momentum,eq:stokes:weak form:pressure}. Similar to \Cref{sec:linear elasticity:reliability}, it is enough to assume here that $\bderrvec \in \honevec{\defbd}$ for the tangential gradient $\nabla_t \bderrvec$ to be in $\Ltwovec{\defbd}$, while no further assumptions are necessary for the Neumann derivative.
From integration by parts and using that $b_\domain(\be_u, e_p) = 0$, we find
\begin{align}
\label{eq:stokes:ibp}
    a_\domain(\be_u, \be_u) =  a_\domain(\be_u, \be_u) + b_\domain(\be_u, e_p) = \int_{\partial \domain}(\stress(\be_u) \bn - e_p \bn) \cdot \be_u \dd{s}.
\end{align}
As for the linear elasticity problem, we must study the Neumann operator on the right-hand side of \cref{eq:stokes:ibp} to derive energy-norm estimates.
To that end, we define the subspace 
\begin{align*}
\mathcal{H}(\domain) \coloneqq \{ (\bs, q) \in \hdivtensor{\domain} \times \Ltwo{\domain}\mid -\nabla\cdot\bs + \nabla q \in \Ltwovec{\domain} \},
\end{align*}
equipped with the norm
\begin{align*}
    ||(\bs, q)||_{\mathcal{H}(\domain)} \coloneqq \Ltwonorm{\bs}{\domain} + \Ltwonorm{q}{\domain} + \Ltwonorm{-\nabla\cdot\bs + \nabla q}{\domain}.
\end{align*}
We define the Neumann operator for the Stokes problem, $\stokesop: \mathcal{H}(\domain) \to \htracedualvec{\partial\domain}$ by integration by parts:
\begin{align*}
    \langle\stokesop(\bs, q), \bmu \rangle \coloneqq \int_\domain (-\nabla \cdot \bs + \nabla q) \cdot \boldsymbol{\phi}_{\bmu} \dd{x} + \int_\domain \bs : \strain(\boldsymbol{\phi}_{\bmu}) \dd{x},
\end{align*}
for some $\boldsymbol{\phi}_{\bmu} \in \honevecbd{\domain}{\bmu}{\partial\domain}$.
The continuity of the operator immediately follows from the Cauchy-Schwarz inequality:
\begin{align*}
    \langle \stokesop(\bs, q), \bmu \rangle &\leq \Ltwonorm{-\nabla \cdot \bs + \nabla q}{\domain} \Ltwonorm{\boldsymbol{\phi}_{\bmu}}{\domain}
    + \Ltwonorm{\bs}{\domain} \honeseminorm{\boldsymbol{\phi}_{\bmu}}{\domain} \leq 
    ||(\bs, q)||_{\mathcal{H}(\domain)} \honenorm{\boldsymbol{\phi}_{\bmu}}{\domain}.
\end{align*}
Taking the infimum over all $\boldsymbol{\phi}_{\bmu}\in \honevecbd{\domain}{\bmu}{\partial\domain}$, we obtain
\begin{align*}
    \langle \stokesop(\bs, q), \bmu \rangle
    \leq ||(\bs, q)||_{\mathcal{H}(\domain)} \htraceinfnorm{\bmu}{\partial \domain},
\end{align*}
which proves the continuity with respect to the equivalent natural trace space norm \cref{eq:preliminaries:definition natural trace norm}.
In analogy to the linear elasticity problem, given a subset $\genbd \subset \partial\domain$ with $|\genbd| > 0$, we may define the partial Neumann operator onto $\genbd$ as $\partialstokesop{\genbd}: \mathcal{H}(\domain) \to \htracedbzdualvec{\genbd}$ by
\begin{align*}
    \langle \partialstokesop{\genbd}(\bs, q), \bmu \rangle \coloneqq \langle \stokesop(\bs, q), \bmu^\star \rangle,
\end{align*}
where $\bmu^\star$ denotes the extension by zero of $\bmu$ to $\htracevec{\partial \domain}$.
\begin{remark}
\label{rem:stokes:neumop operator norm}
Similar to the Neumann operator for the linear elasticity problem, if $\genbd$ itself has a boundary, then a result similar to \cref{lemma:size independence of neumann trace operator norm} holds for $\partialstokesop{\genbd}$ asserting that its operator norm does not depend on the size of $\genbd$.
\end{remark}

Using the boundary representation \cref{eq:stokes:ibp} and the Neumann operator above, we have
\begin{align*}
    a_\domain(\be_u, \be_u) = \langle \stokesop(\stress(\be_u), e_p), \be_u \rangle.
\end{align*}
In the Dirichlet-Dirichlet and internal cases, we have $\bderrvec \in \htracedbzvec{\defbd}$ and in particular,
\begin{align}
\label{eq:stokes:dd and internal boundary rep}
    a_\domain(\be_u, \be_u) = \langle \partialstokesop{\defbd}(\stress(\be_u), e_p), \bderrvec \rangle.
\end{align}
In the Dirichlet-Neumann case, we must again resort to the alternative error representation
\begin{align}
\label{eq:stokes:dn boundary rep}
    a_\domain(\be_u, \be_u) = \langle \partialstokesop{\genbd}(\stress(\be_u), e_p), \extop{\defbd}{\genbd}(\bderrvec) \rangle,
\end{align}
where $\defbd \subset \genbd \subset  \interior{\overline{\defbd} \cup \overline{\rmbd_N}}$ is an open neighborhood of $\defbd$ in the Neumann boundary.

Regarding the error in the pressure variable, we first observe that $b_\domain(\bv, e_p) = -a_\domain(\be_u, \bv)$ for any $\bv \in \honevecbd{\domain}{\boldsymbol{0}}{\rmbd_D}$ due to equation \cref{eq:stokes:weak form:momentum} of the weak form. Moreover, the Stokes problem satisfies the inf-sup condition\cite{ern_finite_2021}, whose constant can be bounded away from zero uniformly in the feature size for a uniformly isotropic family of features\cite[Theorem~2.6]{acosta_divergence_2017}, and, therefore, we have
\begin{nalign}
\label{eq:stokes:inf sup pressure estimate}
    \Ltwonorm{e_p}{\domain} &\lesssim \sup_{\bv \in \honevecbd{\domain}{\boldsymbol{0}}{\rmbd_D}}\frac{b_\domain(\bv, e_p)}{\honeseminorm{\bv}{\domain}}
    \\
    &= \sup_{\bv \in \honevecbd{\domain}{\boldsymbol{0}}{\rmbd_D}}\frac{-a_\domain(\be_u, \bv)}{\honeseminorm{\bv}{\domain}} \lesssim \honeseminorm{\be_u}{\domain} \lesssim a_\domain(\be_u, \be_u)^{\frac{1}{2}},
\end{nalign}
Hence, it is enough to estimate $a_\domain(\be_u, \be_u)^{\frac{1}{2}}$. Let us now prove the reliability of the Dirichlet estimators for Stokes' equations:
\begin{theorem}
\label{thm:stokes:dirichlet estimator reliability}
Consider the defeaturing problem for the Stokes problem with one single Dirichlet feature $\feat$. Then, the following holds:
\begin{enumerate}
    \item If the feature $\feat$ is of Dirichlet-Dirichlet type, then the estimator \cref{eq:stokes:dirdir feature estimator} is reliable. That is, if $\overline{\defbd} \cap \overline{\rmbd_D} \neq \emptyset$ and $\overline{\defbd} \cap \overline{\rmbd_N} = \emptyset
    $, then
    \begin{align*}
        \energynorm{(\be_u, e_p)}{\domain} \lesssim \defestdd{\bu_0}{\defbd}.
    \end{align*}
    \item If the feature $\feat$ is of Dirichlet-Neumann type, then the estimator \cref{eq:stokes:dirneum feature estimator} is reliable. That is, if $\overline{\defbd} \cap \overline{\rmbd_N} \neq \emptyset$, then
    \begin{align*}
        \energynorm{(\be_u, e_p)}{\domain} \lesssim \defestdn{\bu_0}{\defbd}.
    \end{align*}
    \item If the feature $\feat$ is internal, then the estimator \cref{eq:stokes:dirint feature estimator} is reliable. That is, if $\overline{\defbd} \cap \overline{(\partial \domain \setminus \defbd)} = \emptyset$ and $\rcirc(\feat) < \dist(c_\feat, \partial \defdomain)$, then
    \begin{align*}
        \energynorm{(\be_u, e_p)}{\domain} \lesssim \defestint{\bu_0}{\defbd}.
    \end{align*}
\end{enumerate}
\end{theorem}

\begin{remark}
The constants $\tilde{c}_{\defbd}$ and $\Bar{c}_{\defbd}$ play the same role as in the Poisson problem; see \cref{rem:poisson:internal constants} for their practical evaluation.
\end{remark}

\begin{proof}
    We begin with the following observations that hold in all three cases: First, we note that in virtue of the PDE \cref{eq:stokes:error pde} satisfied by the defeaturing error, we have
    \begin{nalign}
    \label{eq:stokes:proof:H norm estimate}
    ||(\stress(\be_u), e_p)||_{\mathcal{H}(\domain)} &= \Ltwonorm{\stress(\be_u)}{\domain} + \Ltwonorm{e_p}{\domain}
    \\
    &\lesssim a_\domain(\be_u, \be_u)^{\frac{1}{2}} + \Ltwonorm{e_p}{\domain} = \energynorm{(\be_u, e_p)}{\domain}.
    \end{nalign}
    Second, the inf-sup estimate \cref{eq:stokes:inf sup pressure estimate} implies that
    \begin{align}
    \label{eq:stokes:proof:energy norm estimate}
        \energynorm{(\be_u, e_p)}{\domain}^2 &\leq 2 a_\domain(\be_u, \be_u) + 2 \Ltwonorm{e_p}{\domain}^2 \lesssim 4 a_\domain(\be_u, \be_u).
    \end{align}

    For assertion 1, we start from the error representation formula \cref{eq:stokes:dd and internal boundary rep}. Similarly to the proof of \cref{thm:linear elasticity:dirichlet estimator reliability}, we obtain from \cref{rem:stokes:neumop operator norm} and \cref{lemma:poincare I} that
    \begin{align*}
    a_\domain(\be_u, \be_u) = \langle \partialstokesop{\defbd}(\stress(\be_u), e_p), \bderrvec \rangle
    \lesssim 2 ||(\stress(\be_u), e_p)||_{\mathcal{H}(\domain)} \htracedbzseminorm{\bderrvec}{\defbd}.
    \end{align*}
    Applying \cref{eq:stokes:proof:H norm estimate} to the first and \cref{lemma:interpolation and poincare inequality} to the second term, we find
    \begin{align*}
        a_\domain(\be_u, \be_u) \lesssim 2 \energynorm{(\be_u, e_p)}{\domain} \sqrt{\Ltwonorm{\bderrvec}{\defbd} \Ltwonorm{\nabla_t \bderrvec}{\defbd}}.
    \end{align*}
    Assertion 1 then follows from \cref{eq:stokes:proof:energy norm estimate}.

    For assertion 2, we start from the error representation formula \cref{eq:stokes:dn boundary rep}. Using \cref{eq:stokes:proof:H norm estimate,eq:stokes:proof:energy norm estimate}, we immediately obtain
    \begin{align*}
        \energynorm{(\be_u, e_p)}{\domain}^2
        &\lesssim 4 a_\domain(\be_u, \be_u) = 4 \langle \partialstokesop{\genbd}(\stress(\be_u), e_p), \extop{\defbd}{\genbd}(\bderrvec) \rangle
        \\
        &\lesssim 8  \energynorm{(\be_u, e_p)}{\domain}\left( \htracedbzseminorm{\extop{\defbd}{\genbd}(\bd_\defbd - \avg{\bd_\defbd}{\defbd})}{\genbd} + \vecnorm{\avg{\bd_\defbd}{\defbd}} \htracedbzseminorm{\extop{\defbd}{\genbd}(1)}{\genbd} \right),
    \end{align*}
    where we have used the same argument as for assertion 2 in the proof of \cref{thm:linear elasticity:dirichlet estimator reliability} in the last step. Assertion 2 then follows from the estimates \cref{eq:elasticity:dirneum proof navg estimate,eq:elasticity:dirneum proof avg estimate}.

    For assertion 3, we similarly find from the error representation formula \cref{eq:stokes:dd and internal boundary rep,eq:stokes:proof:H norm estimate,eq:stokes:proof:energy norm estimate} that
    \begin{align*}
        \energynorm{(\be_u, e_p)}{\domain}^2
        &\lesssim 4 a_\domain(\be_u, \be_u) = 4 \langle \partialstokesop{\defbd}(\stress(\be_u), e_p), \bderrvec \rangle
        \\
        & \lesssim  4\energynorm{(\be_u, e_p)}{\domain} \left(
        \htraceinfnorm{(\bd_\defbd - \avg{\bd_\defbd}{\defbd})^\star}{\partial \domain} + \vecnorm{\avg{\bd_\defbd}{\defbd}}\htraceinfnorm{1^\star}{\partial \domain}
        \right ).
    \end{align*}
    Assertion 3 then follows in the same way as assertion 3 of \cref{thm:linear elasticity:dirichlet estimator reliability}.
\end{proof}

\section{Goal-oriented estimate for the pressure variable of the Stokes problem}
\label{sec:stokes:pressure qoi}
Consider a quantity of interest $L \in Q(\domain)'$ in the dual space of the pressure variable together with a defeatured counterpart $L_0 \in Q_0(\defdomain)'$ satisfying \cref{assumption:defeatured qoi}. By abuse of notation, the exact and defeatured dual problems can respectively be written as
\begin{align}
\label{eq:stokes:exact dual problem for pressure qoi}
\begin{cases}
    -\nabla \cdot \stress(\bz) + \nabla \zeta = 0 & \text{ in } \domain,\\
    \nabla \cdot \bz = L & \text{ in } \domain,\\
     \bz = \boldsymbol{0} & \text{ on } \rmbd_D,\\
    \stress(\bz) \bn - \zeta \bn = \boldsymbol{0} & \text{ on } \rmbd_N,
\end{cases}
\end{align}
and
\begin{align}
\label{eq:stokes:defeatured dual problem for pressure qoi}
\begin{cases}
    -\nabla \cdot \stress(\bz_0) + \nabla \zeta_0 = 0 & \text{ in } \defdomain,\\
    \nabla \cdot \bz_0 = L_0 & \text{ in } \defdomain,\\
    \bz_0 = \boldsymbol{0} & \text{ on } \dirrmbd,\\
    \stress(\bz_0) \bn - \zeta_0 \bn = \boldsymbol{0} & \text{ on } \neumrmbd.
\end{cases}
\end{align}
Note that the solution pairs $(\bz, \zeta)$ and $(\bz_0, \zeta_0)$ are different from the ones in \cref{sec:stokes}. Moreover, if $\rmbd_N = \emptyset$, we must have $L \in (L_0^2(\domain))'$ and analogously for $L_0$ for compatibility.

The corresponding corrector term is defined by
\begin{align*}
    R_L^{\defbd, p}(\bu_0, p_0, \bz_0, \zeta_0)  &\coloneqq -\sum_{\feat \in \neumfeatset} \int_{\defbd^\feat} \bd_{\defbd^\feat} \cdot (\bz_0)_{|\defbd^\feat} \dd{s}
    \\
    &+ \sum_{\feat \in \dirdirfeatset} \int_{\defbd^\feat} (\stress(\bz_0) \bn - \zeta_0 \bn) \cdot \bd_{\defbd^\feat} \dd{s}
    \\
    &+ \sum_{\feat \in \dirintfeatset} \int_{\defbd^\feat} (\stress(\bz_0) \bn - \zeta_0 \bn) \cdot \bd_{\defbd^\feat} \dd{s}
    \\
    &+ \sum_{\feat \in \dirneumfeatset} \int_{\defbd^\feat} (\stress(\bz_0) \bn - \zeta_0 \bn) \cdot \bd_{\defbd^\feat} \dd{s}.
\end{align*}
Comparing the corrector term for a QoI in the pressure variable to the one for a QoI in the velocity variable in \Cref{sec:stokes}, we observe that structurally they only differ by a global change of sign. However, they are not evaluated on the same dual solutions, so numerically they are not equal up to the sign.

The goal-oriented estimate can then be obtained in the exact same manner as for Poisson, linear elasticity and the velocity variable in the Stokes problem.